\documentclass[ reqno]{amsart}
\usepackage{amsfonts}
\usepackage{}
\usepackage{amssymb}
\usepackage{bbm}
\usepackage{amsmath}
\usepackage{amsfonts,amsmath, amsthm, amssymb, amscd, amsfonts, mathrsfs}
\usepackage[all]{xy}
\usepackage{txfonts}
\usepackage{hyperref}
\usepackage[T1]{fontenc}

\usepackage{lmodern}
\usepackage{fourier}
\hfuzz=\maxdimen
\numberwithin{equation}{section}
\newtheorem*{theorem*}{Theorem A}
\newtheorem*{remark*}{Remark }
\newtheorem{lemma}{Lemma}[section]
\newtheorem{proposition}[lemma]{Proposition}
\newtheorem{remark}[lemma]{Remark}

\newtheorem{theorem}[lemma]{Theorem}

\newtheorem{corollary}[lemma]{Corollary}

\baselineskip 20pt \textwidth 18cm  \theoremstyle{plain}

\newcommand{\Hom}{\operatorname{Hom}}

\newcommand{\Span}{\operatorname{Span}}

\newcommand{\Res}{\operatorname{Res}}

\newcommand{\Ha}{\operatorname{H}}
\newcommand{\EU}{\operatorname{EU}}

\newcommand{\Ca}{\mathbb C}

\newcommand{\Gal}{\operatorname{Gal}}
\newcommand{\GL}{\operatorname{GL}}
\newcommand{\GU}{\operatorname{GU}}
\newcommand{\Oa}{\operatorname{O}}
\newcommand{\GSp}{\operatorname{GSp}}

\newcommand{\SL}{\operatorname{SL}}
\newcommand{\SO}{\operatorname{SO}}
\newcommand{\Sp}{\operatorname{Sp}}

\newcommand{\Ua}{\operatorname{U}}
\newcommand{\M}{\operatorname{M}}

\newcommand{\GO}{\operatorname{GO}}

\newcommand{\diag}{\operatorname{diag}}

\newcommand{\id}{\operatorname{Id}}
\newcommand{\nnn}{\operatorname{N}}

\newcommand{\tr}{\operatorname{Tr}}
\newcommand{\Tr}{\operatorname{Tr}}
\newcommand{\Trd}{\operatorname{Trd}}
\newcommand{\Nrd}{\operatorname{Nrd}}

\begin{document}
\title{Splitting metaplectic cover groups}

\author{Chun-Hui Wang}
\address{School of Mathematics and Statistics\\Wuhan University \\Wuhan, 430072,
P.R. CHINA}
\email{cwang2014@whu.edu.cn}
\keywords{ Metaplectic group, Theta correspondence, Weil representation}
\thanks{Research  supported by  NSFC grant \# 11501420.}
\maketitle

\setcounter{tocdepth}{1}
\setcounter{secnumdepth}{3}

\begin{abstract}
If $(G_1, G_2)$ is a dual reductive pair of type I in $\Sp(W)$, it is known that the degree $8$ metaplectic cover of $\Sp(W)$ splits over $G_1G_2$, with one obvious exception. In this paper we replace $G_1G_2$ by a larger subgroup obtained via similitude groups, and show that the degree $8$ metaplectic cover  splits, with the same obvious exception.
\end{abstract}

\section{Statement of the results}\label{statementofresult}
Let  $F$ be  a non-archimedean local field of  residue characteristic different from $2$, and let $D$ be a division algebra over $F$ with an involution $\tau$ such that  $F$ consists of all $\tau$-fixed points of $D$.  We define $(W, \langle, \rangle)$ as a symplectic space over $F$ of dimension $2n$ with a tensor product decomposition
$$W=W_1 \otimes_D W_2, \qquad \langle, \rangle=\Trd_{D/F}\big( \langle, \rangle_1 \otimes \tau(\langle, \rangle_2)\big)$$
where $(W_1, \langle, \rangle_1)$ is a right $\epsilon_1$-hermitian space over $D$ and $ (W_2, \langle, \rangle_2)$ is a left $\epsilon_2$-hermitian space over $D$ with  $\epsilon_1 \epsilon_2=-1$. We  let $\Ua(W_i)$ be the group of isometries of $(W_i, \langle, \rangle_i)$, and $\GU(W_i)$  the group of similitudes  of $(W_i, \langle, \rangle_i)$. In \cite[p.15]{MVW}, it was shown that  except the case $\epsilon_1=1, \epsilon_2=-1$, and $W_2=$ the  quaternion algebra over $F$,  the pair $\big(\Ua(W_1), \Ua(W_2)\big)$ is the so-called irreducible \emph{dual reductive pair}  of type I in the sense of Howe.

 Denote by $\mu_8$ the cyclic group of the roots of unity in $\Ca$ of order $8$. To a non-trivial element $[c_{Rao}]$ of order $2$ in the measurable cohomology group $\Ha^2(\Sp(W), \mu_8)$ is associated a central extension
   $$0 \longrightarrow \mu_8 \longrightarrow \overline{\Sp(W)} \longrightarrow \Sp(W) \longrightarrow 1$$
   of $\Sp(W)$ by $\mu_8$. Let     $\overline{\Ua(W_i)}$ be the degree $8$ metaplectic cover  of $\Ua(W_i)$ induced  by $\overline{\Sp(W)}$. Then  the following  result about the above irreducible dual reductive pair was drawn from {\cite[Chapitre 3, p. 51 ]{MVW}}(written by Vign\'eras).
\begin{theorem}\label{scindagedugroupeR0}
The exact sequence   $0 \longrightarrow \mu_8 \longrightarrow \overline{\Ua(W_1)} \longrightarrow \Ua(W_1) \longrightarrow 1 $ splits, except for $W_1$ being symplectic and $W_2$ being orthogonal of odd dimension.
\end{theorem}
We remark that  Kudla discussed   the explicit splitting group extensions of  dual reductive pairs in \cite{Kud3}. His results  play a significant role in the study of the classical  theta correspondences. In this paper we shall generalize the above result to a larger subgroup of $\Sp(W)$.  We define
$$\Gamma:= \left\{ (g_1, g_2)\mid g_1 \in \GU(W_1), g_2 \in \GU(W_2) \textrm{ such that } \lambda_1(g_1)\lambda_2(g_2)=1\right\}, $$
 where  $\lambda_1$, $\lambda_2$  are the similitude characters from $\GU(W_1)$ , $\GU(W_2)$  to $F^{\times}$ respectively. Accordingly there exists a canonical map
 $\iota : \Gamma \longrightarrow \Sp(W)$.  Now let  $\overline{\Gamma}$ be the degree $8$ metaplectic cover  of   $\Gamma$  induced by $\overline{\Sp(W)}$. By using \cite{MVW}'s approach we derive  the   following coherent result analogy of Theorem \ref{scindagedugroupeR0}.
\begin{theorem*}\label{scindagedugroupeRmain}
 The exact sequence  $1 \longrightarrow \mu_{8} \longrightarrow \overline{\Gamma} \longrightarrow \iota (\Gamma)\longrightarrow 0$
 splits, except  when the irreducible  dual  reductive pair is a symplectic-orthogonal type and the orthogonal vector space over $F$   is of odd dimension.
\end{theorem*}
\begin{remark*}
The major regular part of this result is known to the specialist. However here we deal  it more  systematically and  use  much   cohomology theory.  
 \begin{itemize}
 \item[(1)] When $D=F$, $W_1$ is a symplectic vector space over $F$, and $W_2$ is an   orthogonal vector space over $F$ of even dimension,  the result is compared with  Robert's result \cite[Proposition 3.3.]{Ro}.  
 \item[(2)]  When $D=$ the quaternion algebra over $F$, and  $W_1$ is a  Hermitian vector space over $D$ of even dimension, and $W_2$ is a skew-Hermitian vector space over $D$,  the  result is due to  Gan, see \cite[Sections 2- 3]{Gan}.  

\end{itemize}
\end{remark*} 
 The explicit behavior  of  the cohomology  class  $[c_{Rao}]$ in $\Ha^2(\Sp(W), \mu_8)$ has been  investigated   by Rao  \cite{Rao}, and by Perrin \cite{Per} (see also the comprehensive note \cite{Kud2} written by Kudla). To control the restriction of $[c_{Rao}]$ to $\Gamma$, we use a long exact sequence of cohomology groups obtained by  inflation-restriction from the exact sequence of groups $1 \longrightarrow \Ua(W_1) \times \Ua(W_2) \longrightarrow \Gamma \longrightarrow \Gamma /{(\Ua(W_1) \times \Ua(W_2) )} \longrightarrow 1$.    That exact sequence of cohomology is stated in section  \ref{TheHochschildSerreSpectralSequence}, and section \ref{reductions} deduces a criterion for the splitting of $\overline{\Gamma}$. The criterion involves several conditions, which are verified in the next sections  \ref{TheproofofthemaintheoremI} to \ref{TheproofofthemaintheoremVI}. The verification requires us to use results about the commutator subgroups  of  unitary groups over $D$. So in section \ref{notations} we  recall the classification of the skew hermitian spaces over a $p$-adic division algebra $\mathbb{H}$ (\emph{cf}. \cite{Sa}, \cite{Tsuk})  and the fine structure of the norm one subgroup of   $\mathbb{H}^{\times}$ (\emph{cf}. \cite{Ri}).

\section{Notation and preliminaries }\label{notations}
\subsection{Notation and conventions}\label{notationss}
We  fix the following notations and assumptions for the  whole paper:
\begin{itemize}
\item $F$: a non-archimedean local field of  residue characteristic different from $2$;
\item $E$: a quadratic field extension of $F$;
\item $\mathbb{H}$: the unique quaternion algebra over $F$, up to isometry;
\item $\Nrd, \Trd$: the reduced norm,  resp. trace of $\mathbb{H}$;
\item $\mathbb{H}^0$: the subspace  of elements of pure quaternion of $\mathbb{H}$;
\item $\mathfrak{D}$: the ring of integers of $\mathbb{H}$ consisting of the elements $\mathbbm{d} \in \mathbb{H}$ such that $|\Nrd(\mathbbm{d})|_F \leq 1$;
\item $\mathfrak{P}$: the maximal ideal of $\mathfrak{D}$;
\item $\Ua_{\mathbb{H}}=\left\{ \mathbbm{d} \in \mathbb{H} \mid |\Nrd(\mathbbm{d})|_F=1\right\}$;
\item $\mathbb{SL}_1(\mathbb{H})=\left\{ \mathbbm{d}\in \mathbb{H}^{\times}\mid \Nrd(\mathbbm{d})=1\right\}$;
\item $k_{\mathbb{H}}$: the residue field of $\mathbb{H}$;
\item $\mathbbm{e}_{-1}$: an  element  of $\mathbb{H}^{\times}$ with reduced norm $-1$;
\item $\{1, \xi, \omega, \xi\omega\}$ is a fixed standard basis of $\mathbb{H}$ such that $\xi\omega=-\omega\xi$, and $F(\xi)/F$ is an unramified extension of degree $2$;
\item The ring of integers of a local field  $K$ will be denoted by  $\mathfrak{O}_K$, its  unique maximal ideal by  $\mathfrak{p}_K$,  its residue field by $k_K$, its group of units by $\Ua_K$;
\item $\mathbb{SL}_1(K)=\left\{ d\in K^{\times}\mid \nnn_{K/F}(d)=1\right\}$, for a finite field extension $K/F$;
\item $(D, \tau)$:  a division algebra over $F$ with an involution $\tau$ such that  $F$ consists of all $\tau$-fixed points of $D$; it has  one of the following forms: (1) $D=F, \tau=\id$; (2) $D=E$,  $\tau$=the canonical conjugation; (3) $D= \mathbb{H}$, $\tau=$ the canonical involution;
\item $(H, \langle, \rangle)$: a right (resp. left) $\epsilon$-hermitian  hyperbolic plane  over $D$, defined as $\langle (d_1,
d_1'), (d_2, d_2')\rangle= \tau(d_1) d_2' +
\varepsilon \tau(d_1') d_2$,
 \big(resp.  $\langle (d_1, d_1'),
(d_2, d_2')\rangle= d_1 \tau(d_2') + \varepsilon d_1' \tau(d_2)\big)$, for $d_1, d_2, d_1', d_2' \in D$;
\item $\big(\mathbb{H}[\mathbbm{i}],  \langle, \rangle\big)$: a right (resp. left) skew hermitian vector space over $\mathbb{H}$ of dimension $1$, defined as $\langle d_1,
d_1'\rangle= \tau(d_1) \mathbbm{i}d_1'$ (resp.  $\langle d_1, d_1'\rangle= d_1 \mathbbm{i}\tau(d_1')$), for $d_1,  d_1' \in \mathbb{H}$, where $0 \neq \mathbbm{i} \in \mathbb{H}^0$;
\item $(W, \langle,\rangle)$, $(W_i, \langle, \rangle_i)$, $\epsilon_i$, $\Ua(W_i)$, $\Gamma$, and $\overline{\Ua}(W_i)$, $\overline{\Gamma}$, $\mu_8$, etc.: introduced above;
\item  Without confusion, sometimes we write  a matrix in its block form,     and write for instance $\GL_3(M_2(F))$ for $\GL_6(F)$.
\end{itemize}

\subsection{Preliminaries on unitary groups}\label{thehyperbolicunitarygroups}
In this subsection, we reviewed  some results of unitary groups and their commutator subgroups which are indispensable. Our main references are the books written by  Scharlau \cite{Scha} and Hahn-O'Meara \cite{HOM}. We also do benefit from the articles \cite{Ri}, \cite{Tsuk}. Let us first rephrase  one result from \cite[p.7]{MVW}  concerning about the classification of  anisotropic  $\epsilon$-hermitian spaces over $D$.
\begin{theorem}\label{vigneraslemma}
  Up to isometry,
\begin{itemize}
\item[-]  an anisotropic quadratic  vector space over $F$ has one of  the following forms:
\begin{itemize}
 \item[(i)] $F[a]$,  for $ a \in F^{\times} \textrm{ modulo } (F^{\times})^2$, with the canonical form  $x \longmapsto ax^2$,  $x\in F$;
 \item[(ii)]  $F_1[a]$, for any  quadratic field extension $F_1$ of $F$,   $ a \in F^{\times} \textrm{ modulo } \nnn_{F_1/F}(F_1^{\times})$ with the form $x \longmapsto a\nnn_{F_1/F}(x)$,  $x\in F_1$;
 \item[(iii)] $\mathbb{H}^0[a]$, for $a \in F^{\times} \textrm{ modulo } (F^{\times})^2$,  with the form $x \longmapsto \tau(\mathbbm{x})a \mathbbm{x}$, $\mathbbm{x} \in \mathbb{H}^0$;
  \item[(iv)] $\mathbb{H}$, with the form $\mathbbm{x} \longmapsto \Nrd(\mathbbm{x})$, $\mathbbm{x} \in \mathbb{H}$.
  \end{itemize}
\item[-] an anisotropic hermitian  vector space over $E$ has one of the following forms:
\begin{itemize}
 \item[(i)] $E[a]$,  for $ a \in F^{\times} \textrm{ modulo } \nnn_{E/F}(E^{\times})$, with the form $(x, y)\longmapsto a\tau(x)y$,  $x, y\in E$;
 \item[(ii)] $\mathbb{H}$,  with the form $(\mathbbm{x}, \mathbbm{y}) \longmapsto \Tr_{\mathbb{H}/E}\big( \tau(\mathbbm{x})\mathbbm{y}$\big), $\mathbbm{x}, \mathbbm{y} \in \mathbb{H}$.\footnote{Actually, we use the fact that a (separable) quadratic field extension $E$ of $F$ can be embedded  in $\mathbb{H}$ (\emph{cf}. \cite[p. 326, Proposition]{BushH} ).  The precise trace map is described in Section \ref{Theanisotropicunitary groupI}.}
  \end{itemize}
\item[-] an anisotropic right hermitian  vector space over $\mathbb{H}$ has   the following form: $\mathbb{H}$, with the form $(\mathbbm{x}, \mathbbm{y}) \longmapsto \tau(\mathbbm{x})\mathbbm{y}$, $\mathbbm{x}, \mathbbm{y} \in \mathbb{H}$.
\end{itemize}
\end{theorem}
Suppose   $(V, \langle, \rangle)$ is  an anisotropic hermitian space over $E$. By Hilbert's Theorem 90,  we can  take an element $ \mu$ of $E^{\times}$ such that $\overline{\mu}/{\mu}=-1$.
  Multiplication of $\langle, \rangle$ by $\mu$
  will give a skew hermitian form $\mu \langle, \rangle$ on $V$, so  in analogy with  Theorem \ref{vigneraslemma}, we have
\begin{proposition}\label{propositionforskewher}
Up to isometry,
\begin{itemize}
\item[-]an anisotropic skew hermitian space over $E$ has one of the following forms:
\begin{itemize}
 \item[(i)] $E[a]$,  for $ a \in F^{\times} \textrm{ modulo } \nnn_{E/F}(E^{\times})$, with the form $(x, y)\longmapsto a\mu\tau(x)y$,  $x, y\in E$,
 \item[(ii)] $\mathbb{H}$,  with the form being given as $ (\mathbbm{x}, \mathbbm{y}) \longmapsto \mu \Tr_{\mathbb{H}/E}\big(\tau(\mathbbm{x})\mathbbm{y}\big)$, $\mathbbm{x}, \mathbbm{y} \in \mathbb{H}$;
\end{itemize}
\item[-] an anisotropic right skew hermitian space   over $\mathbb{H}$ has one of the following  forms:
\begin{itemize}
\item[(i)] $\mathbb{H}[\mathbbm{i}]$, for $\mathbbm{i}=\xi, \varpi, \xi \varpi$,
\item[(ii)] $ \mathbb{H}[\mathbbm{i}] \oplus \mathbb{H} [\mathbbm{j}]$,  for $(\mathbbm{i},\mathbbm{j})=(\xi, \mathbbm{e}_{-1}\varpi), (\xi, \mathbbm{e}_{-1}\xi\varpi)$ or $(\varpi, \mathbbm{e}_{-1}\xi\varpi)$,
\item[(iii)] $ \mathbb{H}[\xi]\oplus \mathbb{H} [\varpi]\oplus \mathbb{H}[\mathbbm{e}_{-1}\xi \varpi ] $,

\end{itemize}
where $\mathbbm{e}_{-1}$ is an element of $F(\xi)$ with norm $-1$.
\end{itemize}
\end{proposition}
\begin{proof}
The second assertion is due essentially to  Tsukamoto (see \cite[Theorem 3]{Tsuk}).
\end{proof}

\subsubsection{The hyperbolic unitary groups } Suppose  that $(V=H\oplus V_1, \langle, \rangle)$ is a right $\epsilon$-hermitian space over $D$ of dimension $n$ composed by  a hyperbolic plane $H$ and a subspace $V_1$.
For two vectors $u, v $ in $V$ with $\langle u, u\rangle=0=\langle u, v\rangle$, and  any $d$ in the coset $ \tfrac{1}{2}\langle v, v\rangle+ \mathcal{S}_D$ of $D/{\mathcal{S}_D}$, where
$\mathcal{S}_D=\left\{ s-\epsilon \tau(s)\mid s\in D\right\}$, we define the so-called  \emph{Eichler transformation} related to $u, v, d$ as follows(\emph{cf}. \cite[p. 214]{HOM}):
$e_{u, v, d}(x)=x+ \epsilon u \langle v, x\rangle-(v+ \epsilon u d) \langle u, x\rangle.$
 Let $\EU_V(D)\subseteq \Ua(V)$ be the  group generated  by all above Eichler transformations of $V$.
\begin{theorem}[{\cite[pp.333-335]{HOM}}]
 $[\EU_V(D), \EU_V(D)]=\EU_V(D)$, and  $\Ua(V)=\Ua(H) \cdot \EU_V(D)$.
\end{theorem}
As a consequence we obtain:
\begin{lemma}\label{canonicalsurjectivemap}
There is a  surjective map $\Ua(H)/{[\Ua(H), \Ua(H)]} \longrightarrow \Ua(V)/{[\Ua(V), \Ua(V)]}$.
\end{lemma}

\begin{lemma}\label{Horth}
Let $\mathfrak{A}=\left\{ \begin{pmatrix}
a&0 \\
 0& \overline{a}^{-1} \end{pmatrix}\vert a\in D^{\times}, \overline{a}=\tau(a)\right\} \simeq D^{\times}$.  Despite of  the case $D=F$, $\epsilon=1$, we have:
 \begin{itemize}
  \item[(1)]  $\Ua(H)= \mathfrak{A}\cdot \EU_{H}(D)$,
  \item[(2)] The image of  $ \frac{\mathfrak{A}}{[\mathfrak{A},\mathfrak{A}]}$ or $\frac{D^{\times}}{[D^{\times}, D^{\times}]}$  in  $\frac{\Ua(H)}{[\Ua(H), \Ua(H)]}$ is full.
  \end{itemize}
\end{lemma}
\begin{proof}
The first statement follows from \cite[p. 263]{Scha}, and the second one is a consequence of Lemma \ref{canonicalsurjectivemap}.
\end{proof}

The main purpose  of this section is to achieve the similar results as Lemma \ref{canonicalsurjectivemap} for each  $\epsilon$-hermitian space $V$ over $D$.  Before attempting  to investigate the problem, let us cite the following results from  Carl Riehm's paper \cite{Ri} involving some subgroups  of  $\mathbb{H}^{\times}$.
  \begin{lemma}\label{hermitiandimension1}
\begin{itemize}
 \item[(1)] $[\mathbb{H}^{\times}, \mathbb{H}^{\times}]=\mathbb{SL}_1(\mathbb{H})$.
 \item[(2)] $[\mathbb{SL}_1(\mathbb{H}),\mathbb{SL}_1(\mathbb{H})]= \mathbb{SL}_1(\mathbb{H}) \cap (1 + \mathfrak{P}) $.
 \item[(3)] $\mathbb{SL}_1(\mathbb{H})/[\mathbb{SL}_1(\mathbb{H}),\mathbb{SL}_1(\mathbb{H})]$ is isomorphic with the subgroup $\mathbb{SL}_1( k_{\mathbb{H}})$  consisting of elements of $k_{\mathbb{H}}$ with  norm $1$ in  $k_F$.
  \end{itemize}
\end{lemma}
\begin{proof}
See also \cite[p.648]{PrRa2}.
\end{proof}
\begin{remark}\label{thequotient12}
Suppose now that $\mathbb{H}$ is merely a division algebra over its centre $F$ of dimension $d^2$. Then the above results  also hold.
\end{remark}
\begin{proof}
 See \cite[Section 5 and Theorem 7(iii)(2)]{Ri} for the details.
\end{proof}

Now let $\left\{ 1, \mathbbm{i}, \mathbbm{j}, \mathbbm{k}\right\}$ be a standard base of $\mathbb{H}$ such that $\mathbbm{i} \cdot \mathbbm{j}=-\mathbbm{j} \cdot \mathbbm{i}=\mathbbm{k}$, and $\mathbbm{i}^2=-\alpha, \mathbbm{j}^2=-\beta$. Set $F_1=F(\mathbbm{i})$, and $F_2=F(\mathbbm{j})$.
\begin{lemma}\label{skewhermitiandimension1}
Let $(V=\mathbb{H}[\mathbbm{i}], \langle, \rangle)$ be a right skew hermitian space over $\mathbb{H}$ of dimension $1$.  Then $\Ua(V)= \mathbb{SL}_1(F_1)$ and $\GU(V)=\langle F_1^{\times},  \mathbbm{j}\rangle=F_1^{\times} \cup F_1^{\times} \mathbbm{j}$.
\end{lemma}
\begin{proof}
An element $\alpha_0+ \alpha_1 \mathbbm{j} $ with $\alpha_0, \alpha_1 \in F_1$ lies in $ \Ua(V)$ if and only if   $(\overline{\alpha_0} -\alpha_1 \mathbbm{j})\mathbbm{i} (\alpha_0 +\alpha_1\mathbbm{j})=\mathbbm{i}$; this means $\alpha_1=0$ and $\nnn_{F_1/F}(\alpha_0)=1$, or $\alpha_0=0$ and  $\nnn_{F_1/F}(\alpha_1) \mathbbm{j}^2=1$. But the second case contradicts to $ \mathbbm{j}^2\notin \nnn_{F_1/F}(F_1^{\times})$, so $\Ua(V)=\mathbb{SL}_1(F_1)$. Similarly, if $g=\alpha_0 + \alpha_1\mathbbm{j} \in \GU(V)$, then $(\overline{\alpha_0} - \alpha_1\mathbbm{j}) \mathbbm{i}(\alpha_0+\alpha_1\mathbbm{j})=\lambda(g) \mathbbm{i}$ for some suitable $\lambda(g)\in F^{\times}$. By calculation, we see that $\alpha_0=0$  or $\alpha_1=0$, so the last result follows.
\end{proof}
The following result is immediate:
\begin{lemma}\label{skewhermitiandimension1}
Let $(V=\mathbb{H}$, $\langle, \rangle=\Trd)$  be a right hermitian space over $\mathbb{H}$ of dimension $1$.  Then $\Ua(V)= \mathbb{SL}_1(\mathbb{H})$ and $\GU(V)=\mathbb{H}^{\times}$.
\end{lemma}
\subsubsection{The anisotropic unitary group I}\label{Theanisotropicunitary groupI} Let $(V, \langle, \rangle)$ be  an anisotropic hermitian space over $E$.
\begin{lemma}
If $\dim_E(V)=1$, then $\Ua(V, \langle, \rangle)=\mathbb{SL}_1(E)$, and $\GU(V, \langle, \rangle)=E^{\times}$.
\end{lemma}
\begin{proof}
Obviously.
\end{proof}
Now let us begin to discuss the case when $\dim_E(V)=2$. By Theorem \ref{vigneraslemma}, we assume $E=F(\mathbbm{i})$,  with $\mathbbm{i}^2=-\alpha\in F^{\times}$.  By \cite[p.358]{Scha}, we can choose an element $\mathbbm{j}$ in $\mathbb{H}$, such that $\mathbbm{i}\mathbbm{j}=-\mathbbm{j}\mathbbm{i}=\mathbbm{k}$,  $\mathbbm{j}^2=-\beta \in F^{\times}$,   and  $\left\{1, \mathbbm{i}, \mathbbm{j}, \mathbbm{k}\right\}$ forms a standard basis of $\mathbb{H}$.  Then there is a  decomposition  of $E$-vector space:
 $\mathbb{H}=E \oplus \mathbbm{j} E.$
   Let $\Tr_{\mathbb{H}/E}$ denote the canonical projection from $\mathbb{H}$ to $E$ defined by
$\Tr_{\mathbb{H}/E}(e_1 + \mathbbm{j}e_2)=e_1$, for $ e_1,e_2\in E.$
Now, we can define an $E$-hermitian form $\langle, \rangle$ on $\mathbb{H}$ as follows:
$$\langle e_1+\mathbbm{j}e_2, e_1' + \mathbbm{j}e_2' \rangle=\tr_{\mathbb{H}/E}\Big( \overline{(e_1+ \mathbbm{j} e_2)}(e_1' + \mathbbm{j}e_2')\Big)=\overline{e_1} e_1' + \overline{e_2} \overline{\mathbbm{j}}\mathbbm{j} e_2'= \overline{e_1} e_1' + \beta \overline{e_2} e_2'. $$
If given $a, a' \in E$, we then have
$\langle \big( e_1+ \mathbbm{j}e_2 \big) a, \big( e_1' + \mathbbm{j}e_2'\big) a' \rangle = \overline{a}\, \overline{e_1} e_1' a' +  \beta \overline{a}\,\overline{e_2} e_2' a'=\overline{a} \langle e_1+\mathbbm{j}e_2, e_1' + \mathbbm{j} e_2' \rangle a'.$
 Moreover, if $\langle e_1+\mathbbm{j}e_2, e_1+ \mathbbm{j}e_2 \rangle=\Nrd(e_1)+ \beta \Nrd(e_2)=0,$ then $\Nrd(e_1)=\Nrd(e_2)=0$, namely $e_1=e_2=0$. So  $(\mathbb{H}, \langle, \rangle)$ is the unique  anisotropic  space over $E$ of dimension $2$, up to isometry. Under the basis $\left\{ 1, \mathbbm{j}\right\}$ of $\mathbb{H}$, we identify $\Ua(V, \langle, \rangle)$ with the unitary matrix group $\Ua_V(D)$ consisting of elements $G=\begin{pmatrix}
\alpha_{11}&\alpha_{12}\\
 \alpha_{21} & \alpha_{22}\end{pmatrix} \in \GL_2(E)$ such that
 \begin{equation}\label{AnisotropicE}
G^{\ast}\begin{pmatrix}
  1& 0  \\
  0& \beta
 \end{pmatrix} G= \begin{pmatrix}
  1& 0  \\
  0 & \beta
 \end{pmatrix},
\end{equation}
where $\ast: \GL_{2}(E) \longrightarrow \GL_{2}(E)$ is the conjugate transpose operator. By calculation, (\ref{AnisotropicE}) is equivalent to,
\begin{equation}\label{thedefiningequations2}
\begin{split}
\nnn_{E/F}(\alpha_{11}) + \beta \nnn_{E/F}(\alpha_{21})=1, \\
\nnn_{E/F}(\alpha_{12}) + \beta \nnn_{E/F}( \alpha_{22})=\beta,\\
\overline{\alpha_{11}} \alpha_{12} + \beta\overline{\alpha_{21}} \alpha_{22}=0.
\end{split}
\end{equation}
\begin{lemma}\label{theexpression}
\begin{itemize}
\item[(1)] There is a canonical embedding
$\mathbb{SL}_1(\mathbb{H}) \longrightarrow \Ua(V, \langle, \rangle);$
$ \alpha_1 + \mathbbm{j} \alpha_2 \longmapsto  \begin{pmatrix}
  \alpha_1 & -\beta \overline{\alpha_2}  \\
  \alpha_2 & \overline{\alpha_1}
 \end{pmatrix}$.
 \item[(2)] $\Ua(V, \langle, \rangle)$ contains a subgroup $\mathfrak{U}= \left\{\begin{pmatrix}
  1 & 0  \\
  0 & u
 \end{pmatrix} \mid  u \in \mathbb{SL}_1(E)\right\}$.  Moreover, $\Ua(V, \langle, \rangle) =\left\{ H \cdot A\mid  H\in \mathbb{SL}_1(\mathbb{H}),  A\in \mathfrak{U}\right\}$.
 \end{itemize}
\end{lemma}
 \begin{proof}
The first part of  (2) follows from  the above equations (\ref{thedefiningequations2}). By definition, an element $\alpha_1 + \mathbbm{j} \alpha_2 \in \mathbb{SL}_1(\mathbb{H})$ belongs to $\Ua(V, \langle, \rangle)$, and sends $1$ to $\alpha_1 + \mathbbm{j} \alpha_2$, and $\mathbbm{j}$ to $-\beta \overline{\alpha_2} + \mathbbm{j} \overline{\alpha_1}$, which gives the result (1). Moreover, $G =\begin{pmatrix}
  \alpha_{11} & \alpha_{12} \\
  \alpha_{21} & \alpha_{22}
 \end{pmatrix} \in \Ua(V, \langle, \rangle)$ can be written in the following forms:
 \begin{itemize}
\item[(1)] $G=  \begin{pmatrix}
  \alpha_{11}  & \alpha_{12} \overline{\alpha_{11}} \alpha_{22}^{-1}  \\
  \alpha_{21} & \overline{\alpha_{11}}
 \end{pmatrix}  \cdot  \begin{pmatrix}
 1 & 0  \\
  0 & \overline{\alpha_{11}}^{-1} \alpha_{22}
 \end{pmatrix} $ with $\alpha_{12} \overline{\alpha_{11}} \alpha_{22}^{-1}=-\beta \overline{\alpha_{21}}$ by equations (\ref{thedefiningequations2}), if $\alpha_{11} \neq 0$, and $\alpha_{22} \neq 0$;
\item[(2)] $G=  \begin{pmatrix}
  0  & -\beta \overline{\alpha_{21}}   \\
  \alpha_{21} & 0
 \end{pmatrix}   \cdot  \begin{pmatrix}
 1  & 0 \\
 0 & \alpha_{12} \big( -\beta \overline{\alpha_{21}}\big)^{-1}
 \end{pmatrix}  $ with $\nnn_{E/F}(\alpha_{12} (-\beta \overline{\alpha_{21}})^{-1})=1$ by equations  (\ref{thedefiningequations2}), if $\alpha_{11}= \alpha_{22}=0$ is  possible.
\end{itemize} So the other part of (2) follows.
\end{proof}

\begin{lemma}\label{CommutatorsubgroupofSL1}
$[\Ua(V, \langle, \rangle), \Ua(V, \langle, \rangle)] = [\mathbb{SL}_1(\mathbb{H}), \mathbb{SL}_1(\mathbb{H})] \simeq  \mathbb{SL}_1(\mathbb{H}) \cap (1 +\mathfrak{P})$.
\end{lemma}
\begin{proof}
We prove the result along with the cases given in Proposition \ref{propositionforskewher}.  First of all, $$[\Ua(V, \langle, \rangle), \Ua(V, \langle, \rangle)]= [\mathbb{SL}_1(\mathbb{H}), \mathbb{SL}_1(\mathbb{H})] [\mathbb{SL}_1(\mathbb{H}), \mathfrak{U}] \subseteq \mathbb{SL}_1(\mathbb{H}).$$  In case $\mathbbm{i}= \xi$, $\mathbbm{j}= \varpi$ or $\xi\varpi$, the commutator of two elements $G_1=\begin{pmatrix}
  \alpha_1 & -\beta \overline{\alpha}_2  \\
  u_1\alpha_2 &   u_1\overline{\alpha}_1
 \end{pmatrix}$, $G_2=\begin{pmatrix}
  \alpha_1' & -\beta \overline{\alpha'}_2  \\
  u_1'\alpha_2' &   u_1'\overline{\alpha'}_1
 \end{pmatrix} \in \Ua(V, \langle, \rangle)$, modulo $\mathfrak{P}$,  is presented by
 $$ [G_1, G_2]\equiv \left[ \begin{pmatrix}
  \alpha_1 &0  \\
  u_1\alpha_2 &   u_1\overline{\alpha}_1
 \end{pmatrix},  \begin{pmatrix}
  \alpha_1' &0  \\
  u_1'\alpha_2' &   u_1'\overline{\alpha'}_1
 \end{pmatrix}\right]\equiv  \begin{pmatrix}
  1 &0  \\
  \ast &  1
 \end{pmatrix} \mod \mathfrak{P},$$ which means $[G_1, G_2]\equiv 1 \mod \mathfrak{P} $ in $\mathbb{SL}_1(\mathbb{H})$.  In case $\mathbbm{i}= \varpi$ or $\xi\varpi$, $\mathbbm{j}= \xi$, an element $G=\begin{pmatrix}
  1 &0  \\
  0 &  u_1
 \end{pmatrix} \in \mathfrak{U}$ has the form $ \begin{pmatrix}
  1 &0  \\
  0 &  \pm 1
 \end{pmatrix}  \mod \mathfrak{P}$; hence  the derived  subgroup of $\Ua(V, \langle, \rangle)$ degenerates to that of $\mathbb{SL}_1(\mathbb{H})$.
\end{proof}
\begin{remark}
By Proposition \ref{propositionforskewher}, arguing for the unitary groups of   anisotropic skew hermitian spaces over $E$ eventually reduces to the above cases.
\end{remark}
\subsubsection{The anisotropic unitary groups II}\label{Satakereallyresults}

Let $\big(V=  \mathbb{H} [\mathbbm{i}]\oplus \mathbb{H}[\mathbbm{j}] \oplus \mathbb{H}[\mathbbm{l}], \langle, \rangle\big)$ be a right anisotropic skew hermitian space over $\mathbb{H}$ of dimension $3$ subject to the conditions that (1) $\{ 1, \mathbbm{i}, \mathbbm{j}, \mathbbm{k}=\mathbbm{i}\mathbbm{j}=-\mathbbm{j}\mathbbm{i}\}$ is a standard basis of $\mathbb{H}$; (2) $\mathbbm{i}^2=-\alpha$, $\mathbbm{j}^2=-\beta$, $\mathbbm{l}^2=-\mathbbm{k}^2= \alpha \beta$;\footnote{ We use the different notion from Satake's\cite{Sa}: $\mathbbm{i}^2=-\alpha$, $\mathbbm{j}^2=-\beta$ instead of $\mathbbm{i}^2=\alpha$, $\mathbbm{i}^2=\beta$.}  (3)  $\mathbbm{l}=\mathbbm{i}b_0+\mathbbm{j}c_0+\mathbbm{k}d_0$ for $ b_0, c_0, d_0\in F$, so $b_0^2 \alpha + c_0^2 \beta + d_0^2 \alpha \beta=-\alpha \beta$. The purpose of this sub-section is to review the concrete description  of  the group $\Ua(V, \langle, \rangle)$ by Satake\cite{Sa}, and derive some consequences.

Let us begin with recalling two results in classical groups. Let   $M$ be   a vector space over $F$ of dimension $4$ with  basis $\left\{ x_1, \cdots, x_4\right\}$.
\begin{theorem}[{\cite[Chapitre IV, \S 8]{Di}}]\label{cliffordgroup}
There is an exact sequence
\[
\begin{array}[c]{cccccccccc}
 1&\longrightarrow & F^{\times}       &  \longrightarrow & (F^{\times} \times \GL(M)) &\stackrel{\kappa}{\longrightarrow}  & \GO^+(\bigwedge^2 M, Q) &  \longrightarrow & 1 \\
  &                & t    &    \longmapsto    &  (t^2, t^{-1})  &                       &  &    &
\end{array}
\]
where $ \GO^+(\bigwedge^2 M, Q)=\left\{ g \in \GO(\bigwedge^2 M, Q)\mid \det(g)= \lambda(g)^3\right\}$, and $\lambda:  \GO(\bigwedge^2 M, Q) \longrightarrow F^{\times}$ is the similitude character.
\end{theorem}
\begin{remark}\label{examplebaisi}
If we choose $x_{ij}=x_i \wedge x_j$ for  $1\leq i< j\leq 4$,  to be the basis of $\bigwedge^2 M$, then  $( Q(x_{ij}, x_{kl}))= \diag\Big(\begin{pmatrix}
0 &  1   \\
1 &  0   \end{pmatrix},
\begin{pmatrix} 0& -1   \\
  -1&  0\end{pmatrix},
 \begin{pmatrix}  0&  1   \\
 1&  0\end{pmatrix}\Big)
$. Through the mapping $\kappa$, the action of $g=(a_{ij}) \in \GL_4(F) \simeq \GL(M)$ on $x_k \wedge x_l$ is given by $g \cdot (x_k \wedge x_l)= \sum_{1 \leq  i, j \leq 4} a_{ik} a_{jl} x_{i}\wedge x_{j}$.   Following \cite{Sa}, we denote $g^{(2)}$ to be the matrix in $\GL_6(F)$ such that  $g \cdot (x_{12}, x_{34}; x_{13}, x_{24}; x_{14}, x_{23})= (x_{12}, x_{34}; x_{13}, x_{24}; x_{14}, x_{23}) g^{(2)}$.
\end{remark}

Set $F_1=F(\mathbbm{i})$, $K=F_1\big(\sqrt{-\beta}\big)$,  and $\Gal(F_1/F)= \langle \sigma \rangle$, $\Gal(K/F)=\langle   \sigma, \tau \rangle$. For the above  right skew hermitian space $V$ over $\mathbb{H}$, we let $\textbf{U}_V$ (resp. $\textbf{SU}_V$) be the associated unitary (resp. the special unitary) group  scheme.
\begin{lemma}
$\Ua(V, \langle, \rangle)$ is isomorphic to $\textbf{\emph{SU}}_V(F)$.
\end{lemma}
\begin{proof}
See \cite[p. 21]{MVW} or \cite[Section 2.2]{Gan}.
\end{proof}

By turning to $V_{F_1}= V\otimes_F F_1$, indeed we obtain a right skew hermitian space over the splitting algebra $\M_2(F_1)$ with the  skew hermitian form $\langle,  \rangle_{V_{F_1}}$ induced by the scalar extension. Moreover according to \cite[Section 1]{Sa}, the unitary group $\Ua(V_{F_1}, \langle, \rangle_{V_{F_1}})$ is isomorphic to an orthogonal group $\Oa(W, (, )_{W})$, where  $W = F^{2} \times F^2 \times F^2$, and the form is given by $( (w_1, w_2, w_3), (w_1', w_2', w_3'))_{W} =   w_1^t Q_1w_1'+ w_2^t Q_2 w_2' +w_3^t Q_3 w_3'$,  for $w_i, w_i' \in F^2$,  $Q_1=   \begin{pmatrix}
 0 & -\mathbbm{i} \\
-\mathbbm{i} & 0
\end{pmatrix}$, $Q_2= \begin{pmatrix}
 1 & 0 \\
0 & \beta
\end{pmatrix}$, $Q_3=\begin{pmatrix}
 c_0-d_0\mathbbm{i} & -b_0 \mathbbm{i} \\
-b_0 \mathbbm{i} & \beta(c_0 + d_0 \mathbbm{i})
\end{pmatrix}$.   By \cite[p. 405]{Sa}, we let $P_1= \begin{pmatrix}
 1 & 0 \\
0 & 2\mathbbm{i}
\end{pmatrix}$, $P_2= \begin{pmatrix}
 1 & \sqrt{-\beta} \\
1 & - \sqrt{-\beta}
\end{pmatrix}$, $P_3= \begin{pmatrix}
 -c_0 + d_0 \mathbbm{i}& (b_0 +\sqrt{-\beta})\mathbbm{i}  \\
1 & \frac{(-b_0+ \sqrt{-\beta})\mathbbm{i}}{c_0-d_0\mathbbm{i}}
\end{pmatrix}$, and $P= \diag(P_1, P_2, P_3)$. Let $I=   \begin{pmatrix}
     &  \begin{pmatrix}
-\frac{b_0+ \sqrt{-\beta}}{2} &0\\
0& \frac{\sqrt{-\beta}}{2\mathbbm{i}}
\end{pmatrix}      \\
    \begin{pmatrix}
1 &0\\
0&\frac{\mathbbm{i}(b_0 + \sqrt{-\beta})}{\sqrt{-\beta}}
\end{pmatrix}\end{pmatrix} \in \GL_4(K)$,  $J=   \begin{pmatrix}
     \begin{pmatrix}
0 & -c_0+ d_0\mathbbm{i}\\
1 & 0
\end{pmatrix} &  \\
&  \begin{pmatrix}
0 &-1\\
c_0-d_0\mathbbm{i}  & 0
\end{pmatrix}
\end{pmatrix} \in \GL_4(K) $.
\begin{lemma}[{\cite[Section 3]{Sa}}]\label{thesubgroupofO}
There is an exact sequence
$1 \longrightarrow F_1^{\times} \longrightarrow C^+(W, (, )_W) \stackrel{\kappa}{\longrightarrow} \SO(W, (,)_W) \longrightarrow 1,$
where $C^+(W, (, )_W)=\left\{ (t,g)\mid g\in \GL_4(K), t\in F_1^{\times}\textrm{ such that }t^2 \det(g)=1\textrm{ and }J^{-1}gJ=g^{\tau}\right\}$. The mapping $\kappa$ is defined by $(t,g) \longmapsto t P^{-1} g^{(2)} P$, where $g^{(2)}$ is given in Remark \ref{examplebaisi}.
\end{lemma}

Accord to \cite[p.404]{Sa},  $ \Ua(V, \langle, \rangle) $ is isomorphic to the $\Gal(F_1/F)$-invariant part of $ \SO(W, (, )_W)$.
\begin{proposition}[{\cite[p. 407]{Sa}}]\label{Satake'sresult1}
There exists an exact sequence $1 \longrightarrow F^{\times} \longrightarrow C^+(V, \langle, \rangle) \stackrel{\kappa}{\longrightarrow} \textbf{\emph{SU}}_V(F) \longrightarrow 1$,
where $C^+(V, \langle, \rangle)=\{ (t,g)\mid g\in \GL_4(K), t\in F^{\times} \textrm{ such that } t^2\det(g)=1, J^{-1} g J= g^{\tau}, I^{-1} g I =g^{\sigma}\}$.
\end{proposition}

Let   $\mathbb{D}_4=  \left\{ 0\right\} \cup \left\{ g\in \GL_4(K)\mid J^{-1}g J= g^{\tau}, I^{-1} gI =g^{\sigma}\right\}$. Then  it was shown by Satake in \cite[p.407]{Sa} that $\mathbb{D}_4$ is a division algebra over $F$ of degree $4$, and  $\mathbb{D}_4= \widetilde{K} +  \mathbbm{a}_1\widetilde{K} + \mathbbm{a}_2\widetilde{K}  +  \mathbbm{a}_3\widetilde{K} $ for $\widetilde{K}=\left\{ \widetilde{\alpha} =\diag( \alpha, \alpha^{\tau}, \alpha^{\sigma},  \alpha^{\sigma \tau})\mid \alpha \in K \right\}$ and  $\mathbbm{a}_1=\begin{pmatrix}
    X_{11} &  0 \\
    0  &  X_{22}
    \end{pmatrix}\in\GL_4(K) $, $\mathbbm{a}_2= \begin{pmatrix}
    0 &  Y_{12} \\
    Y_{21}  &  0
    \end{pmatrix}\in \GL_4(K) $,  $\mathbbm{a}_3= \begin{pmatrix}
    0 &  Z_{12} \\
    Z_{21}  &  0
    \end{pmatrix}\in \GL_4(K) $   with $X_{11}=\begin{pmatrix}
    0 &  -(c_0+ d_0 \mathbbm{i})\\
    1  &  0
    \end{pmatrix}$, $X_{22}=\begin{pmatrix}
    0 &  -\frac{\sqrt{-\beta}(c_0+d_0\mathbbm{i})}{\mathbbm{i}(b_0+\sqrt{-\beta})}   \\
    \frac{\mathbbm{i}}{\sqrt{-\beta}} (b_0+\sqrt{-\beta})   &  0
    \end{pmatrix}$, $Y_{21}=\begin{pmatrix}
    1 &  0\\
    0  &  c_0-d_0\mathbbm{i}
    \end{pmatrix}$, $Y_{12}=\begin{pmatrix}
   -\frac{b_0+ \sqrt{-\beta}}{2} &  0\\
    0  &  -\frac{b_0-\sqrt{-\beta}}{2(c_0-d_0\mathbbm{i})}
    \end{pmatrix}$,  $Z_{12}=\begin{pmatrix}
   0 &  \frac{\sqrt{-\beta}}{2\mathbbm{i}}\\
    \frac{\sqrt{-\beta}}{2\mathbbm{i}}  & 0
    \end{pmatrix}$, $Z_{21}=\begin{pmatrix}
   0 &  -1\\
   1& 0
    \end{pmatrix}$ in $\GL_2(K)$. For simplicity, we will identity $\widetilde{K}$ with $K$ henceforth.

Now let $\big(V_1= \mathbb{H}[\mathbbm{j}] \oplus \mathbb{H}[\mathbbm{l}], \langle, \rangle_{V_1}\big)$ be a subspace  of $(V, \langle, \rangle)$.  We denote $D_{F(\mathbbm{i})}= K + \mathbbm{a}_1 K$ to be  the  quaternion algebra over $F(\mathbbm{i})$ endowed with the reduced norm $\Nrd(k_1+ \mathbbm{a}_1 k_2)= k_1^{1+ \tau}+ (c_0-d_0\mathbbm{i})k_2^{1+ \tau}$ for $k_1, k_2 \in K$.
\begin{proposition}[{\cite[p. 409]{Sa}}]\label{Satake'sresult5}
There exists an exact sequence
$1  \longrightarrow F^{\times} \longrightarrow  C^+( V_1, \langle, \rangle_{V_1}) \stackrel{\kappa_1}{\longrightarrow}  \Ua(V_1, \langle, \rangle_{V_1}) \longrightarrow 1,$
where $ C^+( V_1, \langle, \rangle_{V_1}) = \left\{ (t, g) \in C^{+}(V, \langle, \rangle) \mid t\in F^{\times},  g\in D_{F(\mathbbm{i})} \textrm{ and } \Nrd( g)t=1\right\}$.
\end{proposition}

\subsection{}
In the rest of this subsection, we assume that $\mathbbm{j}^2$ is a uniformizer  of $F$.
\begin{lemma}\label{thedecompositionFF2}
Under the conditions at the beginning of Section \ref{Satakereallyresults}, the quotient group $F^{\times}/{(F^{\times})^2}$ is represented by the set $\left\{ 1, -\alpha, \beta, -\alpha\beta\right\}$. \footnote{Notice that if $F$ is  a  non-archimedean  local field  of  residue characteristic $2$, then the cardinality of the quotient group  $F^{\times}/{(F^{\times})^2}$  is much larger than $4$ (\emph{cf}. \cite[p.162, Corollary 2.23]{Lam}).}
\end{lemma}
\begin{proof}
By the second part of  Proposition \ref{propositionforskewher}, and \cite[p. 155, Proposition 2.9]{Lam}, $-1 \in \nnn_{F(\mathbbm{i})/F}(F(\mathbbm{i})^{\times})$, or $ -1 \in \nnn_{F(\mathbbm{k})/F}(F(\mathbbm{k})^{\times})$. In the first case,  $-\alpha=-1 \cdot \alpha \in \nnn_{F_1/F}(F_1^{\times})$.  So $F^{\times}=\langle \nnn_{F_1/F}(F_1^{\times}),-\alpha\beta\rangle$, and $F^{\times}= \nnn_{F_1/F}(F_1)^{\times} \sqcup \big(-\alpha\beta \nnn_{F_1/F}(F_1)^{\times}\big)= (F^{\times})^2 \sqcup \big(-\alpha (F^{\times})^2\big) \sqcup \big(-\alpha\beta (F^{\times})^2\big) \sqcup \big(\beta(F^{\times})^2\big)$. The proof of the second case is similar.
\end{proof}
\begin{lemma}\label{therepsentativesof}
Let $\Upsilon=\left\{(1, 1), (-\alpha, \mathbbm{i}^{-1}), (\beta, \sqrt{-\beta}^{-1}), (-\alpha\beta, \mathbbm{i}^{-1}\sqrt{-\beta}^{-1})\right\}$, and let $\Xi$ be a coset representatives of $\mathbb{SL}_1(\mathbb{D}_{4})/{[\mathbb{SL}_1(\mathbb{D}_{4}), \mathbb{SL}_1(\mathbb{D}_{4})]}$.
  Then the canonical mapping induced by the exact sequence  in  Proposition \ref{Satake'sresult1} from
$T = \{ \omega \varsigma\mid   \omega \in \Xi, \varsigma \in \Upsilon \}$ to $\Ua(V, \langle, \rangle)/{[\Ua(V, \langle, \rangle), \Ua(V, \langle, \rangle)]}$ is surjective.
\end{lemma}
\begin{proof}
Results of Proposition \ref{Satake'sresult1} show that
$$\Ua(V, \langle, \rangle)/{[\Ua(V, \langle, \rangle), \Ua(V, \langle, \rangle)]} \simeq C^+(V, \langle, \rangle)/{\big([ C^+(V, \langle, \rangle), C^+(V, \langle, \rangle)]F^{\times}\big)}$$
 by identifying  $F^{\times}$ with a subgroup of $C^+(V, \langle, \rangle)$ via  the exact sequence there.
  Now let $N=\{ (t,k) \in C^+(V, \langle, \rangle)\mid k\in K^{\times}, t\in F^{\times} \textrm{ such that }\nnn_{K/F}(k)t^2=1\}$
  be a subgroup of $C^{+}(V, \langle, \rangle)$. The field extension $K$ of $F$ contains $F(\mathbbm{i})$ and $F(\sqrt{-\beta})$, so that
 $\nnn_{K/F}(K^{\times}) \supseteq \big(\nnn_{F(\mathbbm{i})/F}(F(\mathbbm{i})^{\times})\big)^2 \cup  \big(\nnn_{F(\sqrt{-\beta})/F}(F(\sqrt{-\beta})^{\times})\big)^2 \supseteq (F^{\times})^2;$
  thus every element $g\in C^+(V, \langle, \rangle)$ can be written in the form $g=(t_1, k_1g_1)$ for $k_1\in K^{\times}$, $t\in F^{\times}$ with $\nnn_{K/F}(k_1)t_1^{2}=1$, and $g_1 \in \mathbb{SL}_1(\mathbb{D}_4)$. This in turn shows that $\mathbb{SL}_1(\mathbb{D}_4)$ is a normal subgroup of $C^+(V, \langle, \rangle)$, and $C^+(V, \langle, \rangle)=\mathbb{SL}_1(\mathbb{D}_4) N$. As a consequence, one obtains
   $$[C^+(V, \langle, \rangle),  C^+(V, \langle, \rangle)] \simeq [ \mathbb{SL}_1(\mathbb{D}_4) N,  \mathbb{SL}_1(\mathbb{D}_4)] =[ \mathbb{SL}_1(\mathbb{D}_4),  \mathbb{SL}_1(\mathbb{D}_4)] \cdot  [ N,  \mathbb{SL}_1(\mathbb{D}_4)]$$
   and
   $$C^+(V, \langle, \rangle)/{\big([C^+(V, \langle, \rangle), C^+(V, \langle, \rangle)] F^{\times}\big)} \simeq \mathbb{SL}_1(\mathbb{D}_4) N/{\big([ \mathbb{SL}_1(\mathbb{D}_4),  \mathbb{SL}_1(\mathbb{D}_4)] F^{\times} [N, \mathbb{SL}_1(\mathbb{D}_4)]\big)}.$$
   By observation, the group $N/{(\mathbb{SL}_1(K) F^{\times})}$  is represented by the set $\Upsilon$;
  this ensures the result.
\end{proof}

\begin{lemma}\label{therepsentativesof2}
Let $\Xi_1$ be a coset representatives of $\mathbb{SL}_1(\mathbb{D}_{F(\mathbbm{i})})/{[\mathbb{SL}_1(\mathbb{D}_{F(\mathbbm{i})}), \mathbb{SL}_1(\mathbb{D}_{F(\mathbbm{i})})]}$. Then the canonical mapping  from $T_1 = \{ \omega \varsigma \mid   \omega \in \Xi_1, \varsigma=(1, 1), (-\alpha, \mathbbm{i}^{-1}), (\beta, \sqrt{-\beta}^{-1}), (-\alpha\beta, \mathbbm{i}^{-1}\sqrt{-\beta}^{-1}) \}$ to $\Ua(V_1, \langle, \rangle_1)/{[\Ua(V_1, \langle, \rangle_1), \Ua(V_1, \langle, \rangle_1)]}$ is surjective.
\end{lemma}
\begin{proof}
The proof is similar  to that of  Lemma \ref{therepsentativesof}.
\end{proof}
\subsection{ Moore cohomology }\label{TheHochschildSerreSpectralSequence}

For the sake of completeness we recall some aspects of spectral sequence of  topological group extensions developed by Moore in \cite{Mo2} and \cite{Mo}. Our  purpose is to extend the classical five inflation-restriction exact sequence to six terms in such case, so that one can use it freely in next sections.

\subsubsection{Moore cohomology}
Let $G$ be a locally profinite group, and let $A$ be a finite abelian  group on which $G$ acts \emph{trivially}. Let $\Ha^{\ast}(G, A)$ be the cohomology groups as defined in \cite{Mo2} by Moore.
Now let $K$ be a normal closed subgroup of $G$. To filter $\Ha^{\ast}(G,A)$ with $G \supseteq K \supseteq 0$, Moore introduced the following standard filtering $\{ L_j\}$ on the complex $C^{\ast}(G,A)$:
\begin{itemize}
\item[(1)] $L_j = \sum_{n=0}^{\infty} L_j \cap C^n(G,A)$.
\item[(2)] $L_j \cap C^n(G,A) =0$, if $j >n$, and $L_j \cap C^n(G,A) =0$, if $j < 0$.
\item[(3)] For $0 \leq j \leq n$,   $L_j \cap C^n(G,A)$ is the group of all elements $f \in C^n(G,A)$ such that $f(s_1, \cdots, s_{n})$ depends on $s_1, \cdots, s_{n-j}$, and the cosets $s_{n-j+1} K, \cdots, s_n K$.
\end{itemize}

It is clear that $\delta(L_j) \subseteq L_j$. Let $Z_r^j$ denote the  preimage of $L_{j+r}$ in $L_j$ and  $E_r^j= Z_r^j/{\big(Z_{r-1}^{j+1} + \delta (Z_{r-1}^{j+1-r})\big)}$. The group $E_{r}^{j,i}$ is just the image of $Z_r^j \cap C^{i+j}(G,A)$ in $E_r^j$.
By definition, one has $E_1^{j,i} \simeq \Ha^{i+j}(L_j/{L_{j+1}})$.  To use the spectral sequence like \cite{HS},  let  $C^j\big(G/K, C^i(K,A)\big)$ be the group of `` normalized ''  $j$-cochains $f$'s on $G/K$ with values in $C^i(K,A)$ subject to the condition that $f( \overline{s_1}, \cdots, \overline{s_j})[t_1, \cdots, t_i]$ defines a \emph{Borel} function from $(G/K)^j \times (K)^i $ to $A$. On $C^j \big(G/K, C^{\ast} (K,A)\big)$, one introduces the natural coboundary operators $\delta_K^{\ast}$, and obtains the $i$-th cohomology group $\Ha^i\big(C^j(G/K$,$ C^{\ast}(K,A)\big)$.
In \cite[ p.48, Lemma 1.1]{Mo2}, it is shown that  $E_1^{j,i}$ is isomorphic to $\Ha^i\big(C^j(G/K, C^{\ast}(K,A))\big)$. Notice that a Borel homomorphism of local compact groups is also continuous, so we have $\Ha^1(K,A) \simeq \Hom(K,A)$, and   $\Ha^1(K,A)$ becomes a topological group when equipped with a  canonical Borel structure.  Along with \cite{HS}  but taking   the Borel structure  into account, Moore proved the following results:
\begin{theorem}[{\cite[p.49 and  p.52, Theorem 1.1]{Mo2}}]
\begin{itemize}
\item[(1)] $E_1^{j,0} \simeq C^j(G/K, A)$, $E_1^{0,i} \simeq \Ha^i(K,A)$.
\item[(2)] $ E_1^{j,1} \simeq C^j\big(G/K, \Ha^1(K, A)\big),$ and $ E_2^{j,1} \simeq  \Ha^j\big(G/K, \Ha^1(K,A)\big)$.
\end{itemize}
\end{theorem}
According to \cite[pp. 52-53]{Mo2}, the composed map
$\Ha^j(G/K, A) \simeq E_2^{j,0} \longrightarrow E_{\infty}^{j,0} \hookrightarrow \Ha^j (G,A)$ is  the inflation from $G/K$ to $G$, $j=1,2$ and the composed map
$\Ha^i(G,A) \twoheadrightarrow E_{\infty}^{0,i} \hookrightarrow E_2^{0,i} \simeq H^i(K,A)$
is the restriction from $G$ to $H$. By \cite[p.130,  Remark]{HS} we have the following inflation-restriction sequence
\begin{equation}\label{5termsequence}
0 \longrightarrow \Ha^1(G/K, A) \stackrel{inf_1}{\longrightarrow} \Ha^1(G,A) \stackrel{res_1}{\longrightarrow} \Ha^1(K,A)^G \stackrel{d_2}{\longrightarrow} \Ha^2(G/K, A) \stackrel{inf_2}{\longrightarrow} \Ha^2(G,A).
\end{equation}
For later use, let us  extend the above long exact sequence  to six terms. Now let  $\Ha^2(G,A)_1$  denote the kernel of the  restriction from $\Ha^2(G,A)$ to $\Ha^2(K,A)$, which is isomorphic to $\Ha^2(L_1)$. Recall that the coimage of   $d_2$ is $E_{\infty}^{2,0}$, which is isomorphic with $\Ha^2(L_2)$. Note that $\Ha^2(L_1)/{\Ha^2(L_2)}$ is isomorphic with $E_{\infty}^{1,1}$. Now $\delta: E_2^{1,1} \longrightarrow E_2^{3,-1}$ is null, so there is an embedding $E_{\infty}^{1,1}  \hookrightarrow E_2^{1,1} \simeq \Ha^1\big(G/K, \Ha^{1}(K,A)\big)$. Hence we conclude that the following exact sequence of six terms holds:
\begin{equation}\label{6termsequence}
0 \longrightarrow \Ha^1(G/K, A) \stackrel{inf_1}{\longrightarrow} \Ha^1(G,A) \stackrel{res_1}{\longrightarrow} \Ha^1(K,A)^{G}\stackrel{d_2}{\longrightarrow} \Ha^2(G/K, A) \stackrel{inf_2}{\longrightarrow} \Ha^2(G,A)_1 \stackrel{p}{\longrightarrow} \Ha^{1}\big(G/K, \Ha^1(K,A)\big).
\end{equation}
Remark that the above sequence coming from  spectral sequence  is  functorial  over the pair $(G ,K)$.
\subsubsection{Explicit expression}\label{Explicitexpression}For convenience, let us describe explicitly the  map $p$ in terms of  cocycles by following \cite{HS}.  Let $[c] \in \Ha^2(G,A)_1$ such that the restriction of the cocycle $c$ to $K \times K$ is trivial. By definition, we have
\begin{equation}\label{2cyccle}
c(s_2, s_3)-c(s_1s_2, s_3)+ c(s_1, s_2s_3) -c(s_1, s_2)=0,  \qquad s_1, s_2, s_3 \in G
\end{equation}
We choose a set of  representatives $\Omega=\{ s^{\ast} \in G\}$ for $G/K$ such that the restriction of the morphism $G \longrightarrow G/K$ to $\Omega$ is a Borel isomorphism, and $1_G \in \Omega$. For $s= s^{\ast} t$, $t\in K$, define $h(s)=c(s^{\ast}, t)$.
Note that
\begin{equation}\label{3cyccle}
\delta_1 h(s, t')= h(t')-h(st')+h(s)=-c(s^{\ast}, tt')+ c(s^{\ast}, t),   \qquad t' \in K
\end{equation}
Consider
$c^{\ast}(s, s_1)=c(s, s_1) +\delta_1 h(s, s_1)$, for $s=s^{\ast} t, s_1 \in G$.
Then $c^{\ast}(s, t')= (c +\delta_1 h)(s, t')$ is zero by (\ref{2cyccle}) and (\ref{3cyccle}). As $\delta_2 c^{\ast}= \delta_2 (c+\delta_1 h)=0$, we obtain
$ c^{\ast} (s^{\ast}, t)-c^{\ast}(s_1 s^{\ast}, t)+ c^{\ast}(s_1, s^{\ast} t)- c^{\ast}(s_1, s^{\ast})=0$. Hence $c^{\ast}(s_1, s)= c^{\ast}(s_1, s^{\ast} t)= c^{\ast}(s_1, s^{\ast})$, meaning that $c^{\ast}(s_1, s)$
depends only on $s_1$ and the coset $sK$. By replacing  $c$ with $c^{\ast}$, the map $p$ is just given by
$p([c^{\ast}]):  sK \longmapsto \big( t  \longrightarrow c^{\ast} (t, s)\big)$, for  $t\in K, s\in G.$

\section{The criterion}\label{reductions}
We keep the notations of Section \ref{notations}. Throughout this section,  we will let $\Lambda_{\Gamma}= \left\{\lambda_1(g_1)=\lambda_2(g_2)^{-1}\mid (g_1, g_2)\in \Gamma\right\}.$
Then there exists a short exact sequence
$0 \longrightarrow  \Ua(W_1) \times \Ua(W_2) \longrightarrow   \Gamma      \stackrel{\lambda}{\longrightarrow}   \Lambda_{\Gamma} \longrightarrow  1$,
where $\lambda$ sends $ (g_1,g_2) $ to  $ \lambda_1(g_1)$. Applying the given map $\iota$ in Section \ref{statementofresult}, we obtain
$0 \longrightarrow \iota\big(\Ua(W_1) \times \Ua(W_2)\big)  \longrightarrow    \iota(\Gamma)   \longrightarrow   \iota(\Gamma)/{\iota\big(\Ua(W_1) \times \Ua(W_2)\big)}   \longrightarrow 1.$
 By an abuse of notation, we write $\iota(\Lambda_{\Gamma})$ for $\iota(\Gamma)/{\iota\big(\Ua(W_1) \times \Ua(W_2)\big)}$ in the following. By  Hochschild-Serre spectral sequence(\emph{cf}. Section \ref{TheHochschildSerreSpectralSequence}),  there exists  the following long exact sequence:
$$0 \longrightarrow \Hom\big(\iota(\Lambda_{\Gamma}), \, \mu_8\big) \longrightarrow \Hom\big(\iota(\Gamma), \, \mu_8\big)  \longrightarrow  \Hom\big(\iota(\Ua(W_1) \times \Ua(W_2)), \, \mu_8\big)^{i(\Gamma)} \longrightarrow
\Ha^2\big( \iota(\Lambda_{\Gamma}), \, \mu_8\big) $$
$$\longrightarrow \Ha^{2}\big(\iota(\Gamma), \, \mu_8\big)_1 \longrightarrow \Ha^1\big( \iota(\Lambda_{\Gamma}), \, \Ha^1( \iota(\Ua(W_1) \times \Ua(W_2)), \, \mu_8) \big) $$
Now let $[c_{Rao}]\in \Ha^2\big(\Sp(W), \mu_{8}\big)$ be the unique nontrivial class of order $2$ and $[c]$  its restriction to $\iota(\Gamma)$. By Theorem \ref{scindagedugroupeR0}, the restriction of $[c]$ to $\iota\big(\Ua(W_1) \times \Ua(W_2)\big)$ is trivial apart from the exceptional case mentioned there. \\

\begin{lemma}\label{onesubgroup}
Assume that  $\Gamma_1$ is a closed subgroup of  $\Gamma$  satisfying the following four conditions:
\begin{itemize}
\item[($\Ca1$)] $\lambda: \Gamma_1 \longrightarrow \Lambda_{\Gamma}$  is surjective;
\item[($\Ca2$)] $\Hom(\iota(\Gamma_1), \mu_{8}) \longrightarrow \Hom\bigg(\iota\big(\Ua(W_1) \times \Ua(W_2)\big)_0, \mu_{8}\bigg)^{\iota(\Gamma_1)}$ is surjective,  where $\iota\big(\Ua(W_1) \times \Ua(W_2)\big)_0= \iota(\Gamma_1) \cap \iota(\Ua(W_1) \times \Ua(W_2))$;
\footnote{ When $D=F$ or $D=\mathbb{H}$, the kernel of $\iota$ is just $\left\{ (1, 1), (-1, -1) \right\}$; when $D=E$, the kernel-$\ker \iota$- belongs to the center of $\Ua(W_1) \times \Ua(W_2)$. If $ \ker \iota$ belongs to $\Gamma_1$,   $\iota(\Gamma_1) \cap \iota\big(\Ua(W_1) \times \Ua(W_2)\big)= \iota\big(\Gamma_1 \cap (\Ua(W_1) \times \Ua(W_2))\big)$. In the general case, one can add $ \ker \iota$ to $\Gamma_1$ to avoid the problem.}
\item[($\Ca3$)] Under the restriction map $\Res: \Ha^2\big( \iota(\Gamma), \mu_8\big) \longrightarrow \Ha^2\big(\iota(\Gamma_1), \mu_8\big)$  the image of $ [c] $  is trivial;
\item[($\Ca4$)] For each $\nu=1,2$, there exists  a set $ \Omega_{\nu}$ of representatives   for  $\iota(\Ua(W_{\nu}) )/[\iota(\Ua(W_{\nu})), \iota(\Ua(W_{\nu}) )]$ such that  the  set $ [\Omega_{\nu}]^{\iota(\Gamma_1)}$ belongs to the parabolic subgroup $ P(Y_{\nu})$  for some Lagrangian vector space $ Y_{\nu}$ of  $ W$.
\end{itemize}
 Then the exact sequence  $1 \longrightarrow \mu_8 \longrightarrow \overline{\Gamma} \longrightarrow \iota(\Gamma) \longrightarrow 1$  splits  over $\iota(\Gamma)$.
\end{lemma}
\begin{proof}
$\iota(\Gamma)$ is identified with $\iota(\Gamma_1) \cdot \iota\big(\Ua(W_1) \times \Ua(W_2)\big)$, so
$\iota(\Gamma_1)/{\iota\big(\Ua(W_1) \times \Ua(W_2)\big)_0} \simeq \iota(\Lambda_{\Gamma})$. Applying the six-term exact sequence  (\ref{6termsequence}) to     the following   commutative diagram
\[
\begin{array}{ccccccccccc}
1 & \longrightarrow & \iota\big(\Ua(W_1) \times \Ua(W_2)\big)_0& \longrightarrow & \iota(\Gamma_1)  & \longrightarrow & \iota( \Lambda_{\Gamma})& \longrightarrow  & 1    \\
  &                 &   \downarrow                     &                 &\downarrow &                 & \downarrow      &                  &      \\
1 & \longrightarrow & \iota\big(\Ua(W_1) \times \Ua(W_2) \big)            & \longrightarrow & \iota(\Gamma)    & \longrightarrow & \iota(\Lambda_{\Gamma}) & \longrightarrow  & 1
\end{array}
\]
 yields  a diagram  of long exact sequences:
\[
\begin{array}{cccccccccccccccc}
 0 &\longrightarrow & \Ha^2\big(\iota(\Lambda_{\Gamma}), \mu_8\big)& \stackrel{\alpha_1}{\longrightarrow }& \Ha^2\big(\iota(\Gamma_1), \mu_8\big)_1  & \longrightarrow  & \Ha^1\Big(\iota(\Lambda_{\Gamma}), \Ha^1\big( \iota(\Ua(W_1)\times\Ua(W_2))_0,\mu_8\big)\Big)   & \longrightarrow &  \cdots  \\
   &                & \uparrow                    &                  &\uparrow    l           &                 & \uparrow
 &                &             \\
 \cdots &\longrightarrow & \Ha^2\big(\iota(\Lambda_{\Gamma}), \mu_8\big)& \stackrel{\alpha}{\longrightarrow }& \Ha^2\big(\iota(\Gamma), \mu_8\big)_1     & \stackrel{p}{\longrightarrow}  & \Ha^1\bigg( \iota(\Lambda_{\Gamma}), \Ha^1\Big(\iota\big(\Ua(W_1) \times\Ua(W_2)\big),\mu_8\Big)\bigg)   & \longrightarrow &  \cdots
 \end{array}
\]
It is  possible to divide the  horizontal arrow $p$ into $p_1$, $p_2$, where
$p_{\nu}: \Ha^2( \iota(\Gamma), \,\mu_8)_1 \longrightarrow \Ha^1\big(\iota(\Lambda_{\Gamma}), \Ha^1(\iota(\Ua(W_{\nu})), \mu_8)\big)$. For each $\nu=1, 2$,
we let  $c_{\nu}$ be  a $2$-cocycle constructed in \cite[p. 55, Th\'eor\`eme]{MVW}, associated to $Y_{\nu}$ and $\psi$, which is a Borel function from $\Sp(W) \times \Sp(W)$ to $\mu_8$. Under  the restriction from $\Ha^2\big(\iota(\Gamma), \, \mu_{8}\big)$ to $\Ha^2\big(\iota(\Ua(W_1) \times \Ua(W_2)), \, \mu_{8}\big)$, the image of $[c_{\nu}]$ is trivial,  meaning that $[c_{\nu}]$ lies in  $\Ha^2\big(\iota(\Gamma), \,  \mu_8\big)_1$. So there is a Borel function $f$ from $\iota\big(\Ua(W_1) \times \Ua(W_2)\big)$ to $\mu_{8}$ such that
$c_{\nu}(t_1, t_2)= f_{\nu}(t_1 t_2) f_{\nu}(t_1)^{-1} f_{\nu}(t_2)^{-1}$,  for  $t_1, t_2 \in \iota\big(\Ua(W_1) \times \Ua(W_2)\big)$.
Note that when  $t_0\in \Omega_{\nu} \subseteq P(Y_{\nu})$, $t \in \iota\big(\Ua(W_1) \times  \Ua(W_2)\big)$, we even have
$f_{\nu}(t_0t)= f_{\nu}(t_0)f_{\nu}(t)= f_{\nu}(tt_0)$. According to \cite[p.42, Definition 1.2]{Mo2}, we can choose a set of representatives $\Delta=\left\{ s^{\ast} \in \iota(\Gamma_1)\right\}$ for $\iota(\Gamma)/{\iota\big(\Ua(W_1) \times \Ua(W_2)\big)}$ subject to the conditions that the restriction of $\iota(\Gamma) \longrightarrow \iota(\Gamma)/{\iota\big(\Ua(W_1) \times \Ua(W_2)\big)}$ to $\Delta$ is a Borel isomorphism, and $\Delta$ contains the identity element of $\iota(\Gamma)$ or $\iota(\Gamma_1)$.  Now let $f_{\nu}$   extend to a Borel  function of $\iota(\Gamma)$ by  taking the trivial value outside $\iota\big(\Ua(W_1) \times \Ua(W_2)\big)$. We replace $c_{\nu}$ with $c_{\nu}'= c_{\nu} \cdot \delta_1 f_{\nu}$,  i.e.  $c_{\nu}'(s_1, s_2)= c_{\nu}(s_1, s_2)f_{\nu}(s_1) f_{\nu}(s_2) f_{\nu}(s_1s_2)^{-1}$ for $ s_1, s_2 \in \iota(\Gamma)$. It is immediate that  $ c_{\nu}'( t, t')=1$ for $t, t' \in \iota\big(\Ua(W_1) \times \Ua(W_2)\big)$. Following Section \ref{Explicitexpression}, we define a Borel function $h$ of $\iota(\Gamma)$ as
$ h(s)= c_{\nu}'(s^{\ast}, t)$, for  $ s=s^{\ast} t \in \iota(\Gamma)$ with $s^{\ast}\in \Delta$, $t\in \iota\big(\Ua(W_1) \times \Ua(W_2)\big)$, and consider  the cocycle $c_{\nu}^{\ast}=c_{\nu}' \cdot \delta_1 h$; then the map $p_{\nu}$ is just given by
$p_{\nu}([c_{\nu}]): s^{\ast}  \longmapsto ( t  \longrightarrow c_{\nu}^{\ast} (t, s^{\ast}))$, for $t\in \iota(\Ua(W_{\nu}) ), s^{\ast}\in \Delta$.
Notice that $p_{\nu}\big([c_{\nu}]\big)(s^{\ast})$ belongs to $\Hom\big(\iota(\Ua(W_{\nu})), \mu_8\big)$, so it depends only on the values at  those $ t_0 \in \Omega_{\nu}$. Now
$c_{\nu}^{\ast}(t_0, s^{\ast})= c_{\nu}'(t_0, s^{\ast}) [\delta_1 h(t_0, s^{\ast})]=c_{\nu}'(t_0, s^{\ast}) c_{\nu}'(s^{\ast}, (s^{\ast})^{-1}t_0s^{\ast})^{-1}=f_{\nu}(t_0)f_{\nu}^{s^{\ast}}(t_0)^{-1}$, a coboundary with respect to
$\Ha^1\big(\iota(\Gamma), \Hom(\iota(\Ua(W_{\nu}) ), \mu_8)\big)$.  Hence under the map $p_{\nu}$, the image of $[c_{\nu}]$ is trivial, and $[c_{1}]=[c_{2}]=\alpha([d])$  for  some $[d] \in \Ha^2(\iota(\Lambda_{\Gamma}), \mu_8)$.
By assumption, $l\big( [c_{\nu}]\big)=0=\alpha_1([d])$; as  $\alpha_1$ is injective,  we get $[d]=0$, and then $[c_{\nu}]=0$,   so the result follows.
\end{proof}

\begin{remark}\label{invariant}
Keep the above  notations. The  result also holds if we replace the above $(\Ca3)$ and $(\Ca4)$ by the following condition:
 \begin{itemize}
 \item[($\Ca 3\tfrac{1}{2})$]  For each $\nu=1,2$, there exists  a set $ \Omega_{\nu}$ of representatives   for  $\iota(\Ua(W_{\nu}) )/[\iota(\Ua(W_{\nu})), \iota(\Ua(W_{\nu}) )]$ such that
 the  set $ [\Omega_{1}]^{\iota(\Gamma_1)}$ belongs to $ P(Y_{1})$  for some Lagrangian vector space $ Y_{1}$ of  $ W$, and the restriction of $[c]$ to the new group generated by $\iota(\Gamma_1)$ and $\iota(\Omega_2)$ is trivial.
 \end{itemize}
\end{remark}
\begin{proof}
Under the above condition, it can also be shown that $p([c])$ is trivial in $ \Ha^1\Big( \iota(\Lambda_{\Gamma}), \Ha^1\big(\iota(\Ua(W_1) \times\Ua(W_2)),\mu_8\big)\Big) $, and the remaining proof  is the same as above.
\end{proof}
In the following Sections \ref{TheproofofthemaintheoremI}---\ref{TheproofofthemaintheoremVI}, we shall prove Theorem A. Our main tool is the above lemma \ref{onesubgroup}, and it reduces to find certain suitable subgroup $\Gamma_1$ of $\Gamma$ satisfying the desired conditions ($\Ca1$)---($\Ca4$), or the variant ones.  In the rest of this paper, we shall frequently come back to these conditions without  further illustration.
\section{The proof of  the main theorem-Part I.}\label{TheproofofthemaintheoremI}
In this section, we  follow the notations of Section \ref{reductions} and  prove Theorem A (\emph{cf}. Section  \ref{statementofresult}) in  one major type of  cases, where either $W_1$ or $W_2$ is  a hyperbolic space over $D$.

\begin{proposition}\label{thespittingcases}
 When $W_1$ or $W_2$ is a hyperbolic space over $D$, Theorem A holds.
 \end{proposition}
\begin{proof}
Without loss of generality,  we assume that $W_2 \simeq n_2H$ for  a hyperbolic plane $H$ over $D$. Let  $H= X\oplus X^{\ast}$ be a complete polarisation, so that  we can  define two sections
 $s^{\pm}: F^{\times} \longrightarrow \GU(W_2); a \longmapsto   \underbrace{t_a^{\pm} \times \cdots \times t_a^{\pm}}_{n_2 }$
 for $t_a^{\pm}=\left(\begin{array}{ccccccc}
\pm 1& 0 \\
0 & \pm a \end{array}\right)\in  \GU(H)$.  Consider the subgroup
$\Gamma_1=\left\{ (g_1,g_2) \in \GU(W_1) \times s^{\pm}(F^{\times}) \mid \lambda_1(g_1)=\lambda_2(g_2)^{-1} \right\}$
 of $\Gamma$. We  identify $\Lambda_{\Gamma}=\Lambda_{\Gamma_1}$,   and fairly have
 $\Gamma_1 \cap \big( \Ua(W_1, \langle, \rangle_1) \times \Ua(W_2, \langle, \rangle_2)\big) = \left\{ (1, 1), (-1, -1)\right\}.$ It is known that  $\Gamma_1$ belongs to a parabolic subgroup $P(W_1 \otimes X^{\ast})$ of $\Sp(W, \langle, \rangle)$,  provided the condition ($\Ca 3$).  By Lemmas \ref{canonicalsurjectivemap}, \ref{Horth},\footnote{Here, we assume that $W_2$ is not an orthogonal space over $F$.}    $P(W_1  \otimes X^{\ast})$ contains a complete set of representatives for $\Ua(W_2, \langle, \rangle_2)/{[\Ua(W_2, \langle, \rangle_2 ), \Ua(W_2, \langle, \rangle_2)]}$, and also $\Ua(W_1, \langle, \rangle_1)$, $\Gamma_1$, so the condition $(\Ca 4)$  holds.  By Lemma \ref{onesubgroup}, the result follows.
 \end{proof}
\begin{corollary}\label{scindagesursymplectique}
If $W_1$ is a symplectic vector space over $F$  and  $W_2$ is an orthogonal vector space over $F$ of even dimension, then  the central extension  $\overline{\Gamma}$  splits over  $\Gamma$.
\end{corollary}
\section{The proof of  the main theorem-Part II.}\label{TheproofofthemaintheoremII}
In this section we shall prove Theorem A in another type of cases.
\begin{proposition}
Let  $D=E$ be a quadratic field extension of $F$. If both $W_1$, $W_2$ are anisotropic vector spaces over $E$, then Theorem A holds.
\end{proposition}

\begin{proof}
Cases $I$ $\&$  $II$: $\dim_{E}(W_1)=1$ or $\dim_E(W_2)=1$. By almost symmetry, we only deal with the first case. Assume $W_1=E(f)$,  for $f=1$ or  $f\in F^{\times} \smallsetminus \nnn_{E/F}(E^{\times})$.
We now   define an $F$-bilinear mapping $\theta$ from $W\simeq E(f)\otimes_EW_2$ to $W_2$ as
 $\theta: W= E(f) \otimes_E W_2 \longrightarrow W_2;$
  $e\otimes w_2 \longmapsto ew_2$, which  further induces  an isometry of symplectic spaces over $F$, when $W_2$ on the right-hand side is endowed with the symplectic form  $\langle, \rangle_{2 ,F}= \Tr_{E/F}(f \overline{\langle, \rangle}_2)$. Then the  composite map
$\iota: \Gamma \hookrightarrow \Sp(W, \langle, \rangle) \simeq \Sp(W_2, \langle, \rangle_{2,F}),$  induced by the above $\theta$,
  implies that $\iota(\Gamma)= \iota \big(\Ua(W_1, \langle,\rangle_1) \times \Ua(W_2, \langle, \rangle_2)\big)$. By Theorem \ref{scindagedugroupeR0}, $\overline{\Gamma}$ splits over  $\Gamma$.\\
Case III: $\dim_{E} (W_1) = \dim_{E} (W_2) =2$. By Theorem \ref{vigneraslemma},  we  assume that $W_1 \simeq \mathbb{H}$ and $W_2 \simeq \mathbb{H}$. Suppose  now $E=F(\mathbbm{i})$ for some $ \mathbbm{i}\in \mathbb{H}^0$, with $\mathbbm{i}^2=-\alpha \in F^{\times}$. By \cite[p. 358]{Scha}, we choose an element $\mathbbm{j}\in\mathbb{H}$, such that $\{1, \mathbbm{i}, \mathbbm{j}, \mathbbm{k}=\mathbbm{i} \mathbbm{j}\}$ forms a standard basis of $\mathbb{H}$, with $\mathbbm{i}^2=-\alpha$ and $\mathbbm{j}^2=-\beta$.
   Let $\Tr_{\mathbb{H}/E}$ denote the canonical projection from $\mathbb{H}$ to $E$ defined by
$\Tr_{\mathbb{H}/E}(e_1 + \mathbbm{j}e_2)=e_1$, for $ e_i\in E$.
Then the $E$-vector space $\mathbb{H}=E \oplus \mathbbm{j} E$, equipped with the  form defined as
$\langle e_1+\mathbbm{j}e_2, e_1' + \mathbbm{j}e_2' \rangle_1=\tr_{\mathbb{H}/E}\big( \overline{(e_1+ \mathbbm{j} e_2)}(e_1' + \mathbbm{j}e_2')\big)= \overline{e_1} e_1' + \beta \overline{e_2} e_2', $ for $e_i, e_i'\in \mathbb{H}$,
will turn into a hermitian space over $E$.  On the other hand we  assume  that the form $\langle, \rangle_2$ on $\mathbb{H}= E\oplus E\mathbbm{j}$ is just given by $-\mathbbm{i} \overline{\langle, \rangle_1}$, i.e., if   $e_1+e_2\mathbbm{j}, e_1' + e_2' \mathbbm{j} \in \mathbb{H}$, we  then  have
$\langle e_1+e_2\mathbbm{j}, e_1' + e_2'\mathbbm{j}\rangle_2= -\mathbbm{i}(e_1\overline{e_1'}+ \beta e_2 \overline{e_2'})= \Tr_{\mathbb{H}/E}\big( -\mathbbm{i} (e_1+e_2 \mathbbm{j}) (\overline{e_1'} + \overline{e_2'\mathbbm{j}})\big)$, for $ e_1, e_1' , e_2, e_2' \in E.$ The skew hermitian  form $\langle, \rangle_{W, E}$ on $W=\mathbb{H}\otimes_E \mathbb{H}$ is now  defined as
$\langle w_1\otimes w_2, w_1' \otimes w_2' \rangle_{W, E}= \langle w_1, w_1'\rangle_1 \overline{ \langle w_2, w_2' \rangle_2} =
\mathbbm{i}\big( \overline{a_1}\overline{b_1} a_1'b_1' + \beta\overline{a_1}\overline{b_2}a_1'b_2'\big)+
 \mathbbm{i}\beta \big( \overline{a_2}\overline{b_1} a_2' b_1' + \beta \overline{a_2} \overline{b_2} a_2' b_2'\big)$,
for $w_1=a_1+ \mathbbm{j} a_2, w_1'=a_1' + \mathbbm{j}a_2' \in \mathbb{H}; w_2=b_1 + b_2\mathbbm{j}, w_2'= b_1' + b_2' \mathbbm{j}\in \mathbb{H}$.
As is easy to verify that if  $\mathbb{H}\oplus \mathbb{H}$  is endowed with the   skew hermitian  form
$\langle, \rangle_{\mathbb{H} \oplus \mathbb{H}, E }= \overline{\langle, \rangle_2} +  \beta \overline{\langle, \rangle_2}$, then the  $E$-linear map
$\theta=\theta_1 \oplus \theta_2: W=\mathbb{H}\otimes_E \mathbb{H} \simeq \big(E \oplus \mathbbm{j}E\big)\otimes_E \mathbb{H} \longrightarrow \mathbb{H} \oplus \mathbb{H};$
$ [(a_1 +  \mathbbm{j}a_2) \otimes (b_1 + b_2\mathbbm{j})] \longmapsto    \Big( a_1b_1 + a_1b_2\mathbbm{j}, a_2b_1 +a_2 b_2 \mathbbm{j}\Big),$
 defines an isometry  from $(\mathbb{H}\otimes_E \mathbb{H}, \langle, \rangle_{W, E})$ to $(\mathbb{H} \oplus \mathbb{H}, \langle, \rangle_{\mathbb{H} \oplus \mathbb{H}, E})$. By  definition, we can embed $E^{\times} $ into $\GU(W_1)$ resp. $\GU(W_2)$ defined   as
$e \cdot( e_1 + \mathbbm{j}e_2):= e_1e + \mathbbm{j} e_2e$ resp. $ e\cdot( e_1 + e_2 \mathbbm{j}):=  e e_1 + e e_2\mathbbm{j}$,  for $ e\in E^{\times};$ both
 multipliers  of $e$ are   the same $\nnn_{E/F}(e)$.  We choose an element  $\mathbbm{e}_{-1} \in \mathbb{H}^{\times}$, with reduced norm $-1$. Now let  $\mathbbm{e}_{-1}\mathbbm{j}$ act  on  $W_1$ as
 $[\mathbbm{e}_{-1}\mathbbm{j}, e_1 + \mathbbm{j}e_2]:= \mathbbm{e}_{-1}\mathbbm{j}\cdot (e_1 + \mathbbm{j}e_2),$
  and on $W_2$ as
$[\mathbbm{e}_{-1}\mathbbm{j}, e_1 + e_2\mathbbm{j}]:= (e_1 + e_2\mathbbm{j}) \cdot \mathbbm{e}_{-1}\mathbbm{j}$
 with multipliers both being  $\Nrd(\mathbbm{e}_{-1}\mathbbm{j})=-\beta$.

 (A) If  $ \mathbbm{i}= \xi$, $\mathbbm{e}_{-1}\in E^{\times}$, we define a subgroup  $\Gamma_1$ of $\Gamma$  as  generated by $(\mathbbm{j}, \mathbbm{j}^{-1})$, $ (e, e^{-1})$, for  all $ e\in E^{\times}$. Then $\lambda(\Gamma_1)=
  \langle \nnn_{E/F}(E^{\times}), \beta \rangle=F^{\times}$, and $\iota(\Gamma_1) \cap \iota\big(\Ua(W_1, \langle, \rangle_1) \times \Ua(W_2, \langle, \rangle_2)\big)=1$. Indeed by definition,  we have $\mathbbm{j} \cdot (a + \mathbbm{j} b)= -\beta b+ \mathbbm{j}a$ and $(a+b\mathbbm{j}) \cdot  \mathbbm{j}^{-1}= b - \frac{a}{\beta}\mathbbm{j}$, so through the above $\theta$ the element $(\mathbbm{j}, \mathbbm{j}^{-1})$ acts on $\mathbb{H} \oplus \mathbb{H}$ as
    $(\mathbbm{j}, \mathbbm{j}^{-1})\bullet  \Big( a_1 + b_1 \mathbbm{j}, a_2 + b_2 \mathbbm{j}\Big)=( -\beta b_2 + a_2  \mathbbm{j}, b_1 -\frac{a_1}{\beta}\mathbbm{j})$. Hence the image of $\Gamma_1$ or $\Gamma$ in $\Sp(W, \langle, \rangle)$ sits in $\Ua(W, \langle, \rangle_{W, E})$. By Theorem \ref{scindagedugroupeR0}, the result holds in this case.

 (B) If  $\mathbbm{i}= \varpi$ or $\xi \varpi$, and $\mathbbm{j}=\xi$, we set $F_2=F(\xi)$. In this case,   $\mathbbm{e}_{-1}=a_{-1} + \mathbbm{j}b_{-1} \in F( \mathbbm{j})^{\times}$, for some $a_{-1}, b_{-1}\in F^{\times}$.  We define a subgroup  $\Gamma_1$ of $\Gamma$  as generated by   $(\mathbbm{e}_{-1}\mathbbm{j}, (\mathbbm{e}_{-1}\mathbbm{j})^{-1})$,  $(e, e^{-1})$, for  all $ e\in E^{\times}$.
Then $\lambda(\Gamma_1)= \langle \nnn_{E/F}(E^{\times}), -\beta\rangle= F^{\times}$, and $\iota(\Gamma_1) \cap \iota\big(\Ua(W_1, \langle, \rangle_1) \times \Ua(W_2, \langle, \rangle_2)\big)=1$. By definition, $\mathbbm{e}_{-1} \mathbbm{j} \cdot (a+ \mathbbm{j}b)= - \beta\mathbbm{e}_{-1} b + \mathbbm{e}_{-1}\mathbbm{j}a$, and $(a+ b \mathbbm{j}) \cdot (\mathbbm{e}_{-1} \mathbbm{j})^{-1}=(a+ b \mathbbm{j}) \cdot \frac{\mathbbm{j}}{\beta} \overline{\mathbbm{e}_{-1}} = -b \overline{\mathbbm{e}_{-1}} + \frac{a \overline{\mathbbm{e}_{-1}}}{\beta} \mathbbm{j}$, so through the above $\theta$, the element $\big(\mathbbm{e}_{-1} \mathbbm{j}, (\mathbbm{e}_{-1} \mathbbm{j})^{-1}\big)$ acts on $\mathbb{H} \oplus \mathbb{H}$ as
 $$\big(\mathbbm{e}_{-1} \mathbbm{j}, (\mathbbm{e}_{-1} \mathbbm{j})^{-1}\big)\bullet [a_1 + b_1 \mathbbm{j}, a_2 + b_2 \mathbbm{j}]
    =[ (b_1 b_{-1}+ b_2 a_{-1})\beta -(a_1 b_{-1}+a_2 a_{-1})\mathbbm{j}, (-b_1 a_{-1}+ \beta b_2 b_{-1}) + (\frac{a_1a_{-1}}{\beta}-a_2 b_{-1})\mathbbm{j}] \bullet \overline{\mathbbm{e}_{-1}}.$$
   By observation, the image of $\Gamma_1$  in $\Sp(W, \langle, \rangle)$ belongs to  $\Ua(W, \langle, \rangle_{W, E})$. By Theorem \ref{scindagedugroupeR0}, the result holds in this case. This  completes the proof.
\end{proof}

In Sections \ref{TheproofofthemaintheoremIII}---\ref{TheproofofthemaintheoremV}, in view of the results of Satake( Section \ref{Satakereallyresults}) we shall prove Theorem A in some quaternionic  cases.

\section{The proof of  the main theorem-Part III.}\label{TheproofofthemaintheoremIII}

In this section we shall follow the notations introduced at beginning  of Section \ref{reductions}. Let $(W_1= \mathbb{H}[\mathbbm{j}] \oplus \mathbb{H}[\mathbbm{l}], \langle, \rangle_1)$ be a right anisotropic shew hermitian space over $\mathbb{H}$ of dimension $2$  such that
(1) $\left\{ 1, \mathbbm{i}, \mathbbm{j}, k=\mathbbm{i}\mathbbm{j}=-\mathbbm{j}\mathbbm{i}\right\}$ is a standard basis of $\mathbb{H}$;
(2) $\mathbbm{i}^2=-\alpha$, $\mathbbm{j}^2=-\beta$, $\mathbbm{l}^2=-\mathbbm{k}^2=\alpha \beta$;
(3) $\mathbbm{l}=b_0\mathbbm{i}+ c_0\mathbbm{j}+d_0\mathbbm{k}$ for some $b_0, c_0, d_0\in F$ satisfying $b_0^2 \alpha+ c_0^2\beta+d_0^2\alpha \beta=-\alpha\beta$.
Let $(W_2=\mathbb{H}, \langle,\rangle_2)$ be a left hermitian space over $\mathbb{H}$ of one dimension with the hermitian form defined by  $\langle \mathbbm{d}, \mathbbm{d}' \rangle_2= \mathbbm{d} \overline{\mathbbm{d}'}$, for vectors $\mathbbm{d}, \mathbbm{d}'\in \mathbb{H}$. Let $(W, \langle, \rangle)= (W_1 \otimes_{\mathbb{H}} W_2, \Trd(\langle,\rangle_1\otimes \overline{ \langle,\rangle}_2))$ be as in Section \ref{reductions}.

 Let $ \{ 1, \xi, \varpi, \xi\varpi\}$ be the fixed standard basis of $\mathbb{H}$ given in Section \ref{notationss}. If  $W_1$ is endowed  with the $F$-symplectic form $\langle, \rangle_{1, F}=\Trd(\langle, \rangle_1)$, then one can check that the canonical mapping
$\theta: W= W_1 \otimes_{\mathbb{H}} W_2 \simeq \big( \mathbb{H}[\mathbbm{j}] \oplus \mathbb{H}[\mathbbm{l}]\big) \longrightarrow W_1 =\mathbb{H}[\mathbbm{j}] \oplus \mathbb{H}[\mathbbm{l}];$
$ w_1 \otimes \mathbbm{d} \longmapsto w_1 \mathbbm{d}, $
defines an $F$-isometry from $(W, \langle, \rangle)$ to $(W_1, \langle, \rangle_{1, F})$.  An element $g \in\Ua(W_2, \langle, \rangle_2)$ now acts  on $W_1$  on the right-hand side by multiplication. In the following,  we shall prove the result closely along the different   cases described in  Proposition \ref{propositionforskewher}.

\subsection{Cases $I$ $\&$ $II$}We assume   $(\mathbbm{i}, \mathbbm{j},
\mathbbm{l})=(\varpi , \xi, \mathbbm{e}_{-1} \xi \varpi )$ or $(\xi
\varpi , \xi, \mathbbm{e}_{-1} \varpi )$, in which
cases we may and do assume  $c_0=0$. Assume that  $\mathbbm{e}_{-1}$ is  the Teichm\"uller representative of an element of $k_{F(\mathbbm{j})}$ in $\mathfrak{O}_{F(\mathbbm{j})}$ with order $2(q+1)$. It then follows that $\Nrd(\mathbbm{e}_{-1})=-1$. For simplicity, we choose  certain $b_0$, $d_0$ at the  beginning such that  $(\mathbbm{e}_{-1})^{-1}=d_0 + \tfrac{b_0}{\beta} \mathbbm{j}$, and $\mathbbm{l}= (\mathbbm{e}_{-1})^{-1} \mathbbm{k}$.
 Let us fix  a symplectic basis $\mathcal{A}_2=\left\{ e_1=-\frac{1}{2\beta}, e_2=\frac{\mathbbm{i}}{2\alpha \beta}; e_1^{\ast}=\mathbbm{j},  e_2^{\ast}=\mathbbm{k}\right\}$  of   the subspace $(\mathbb{H}[\mathbbm{j}],\langle, \rangle_{1,F}) $ and  a symplectic basis $\mathcal{A}_3=\left\{ f_1=\frac{1}{2}, f_2=-\frac{\mathbbm{j}}{2\beta}; f_1^{\ast}=\frac{\mathbbm{l}}{\alpha \beta}, f_2^{\ast}=\frac{\mathbbm{j}\mathbbm{l}}{\alpha \beta}\right\}$ of  the subspace $(\mathbb{H}[\mathbbm{l}],\langle, \rangle_{1,F})$. Let  $U_1=\Span\{e_1, e_1^{\ast}\}$, $U_2=\Span\{ e_2, e_2^{\ast}\}$ be the symplectic vector subspaces of $(\mathbb{H}[\mathbbm{j}], \langle, \rangle_{1,F})
 $ with the symplectic basis $\mathcal{A}_2^{(1)}=\{e_1,e_1^{\ast}\}$, $\mathcal{A}_2^{(2)}=\{e_2, e_2^{\ast}\}$ respectively.

 In those cases, $\Ua(W_2, \langle, \rangle_2) \simeq \mathbb{SL}_1(\mathbb{H})$. By  Lemma \ref{hermitiandimension1},   $\Ua(W_2, \langle, \rangle_2) /{[\Ua(W_2, \langle, \rangle_2), \Ua(W_2, \langle, \rangle_2)]}$ is a cyclic group of order $(q+1)$ generated by $\mathbbm{e}_{-1}^2=(d_0^2-\tfrac{b_0^2}{\beta}) -\tfrac{2b_0d_0}{\beta} \mathbbm{j}$ in $\mathbb{SL}_1(F(\mathbbm{j}))$.
Under the above basis $\mathcal{A}_2^{(1)}$, $\mathcal{A}_2^{(2)}$, and $\mathcal{A}_3$, $\mathbbm{e}_{-1}^2$  acts on  $U_1$, $U_2$ and  $\mathbb{H}[\mathbbm{l}]$ by means of the matrices
 $G_1^{(1)}=\begin{bmatrix}
d_0^2-\tfrac{b_0^2}{\beta} &
-4\beta b_0d_0\\
\tfrac{b_0d_0}{\beta^2} &
d_0^2-\tfrac{b_0^2}{\beta}   \end{bmatrix}
  $, $G_1^{(2)}=
\begin{bmatrix}
d_0^2-\tfrac{b_0^2}{\beta}   &4\alpha\beta b_0d_0  \\
 -\tfrac{b_0d_0}{\alpha \beta^2}   &d_0^2-\tfrac{b_0^2}{\beta} \end{bmatrix}$,  and
   $ G_2=\begin{pmatrix}
 \begin{bmatrix}
d_0^2-\tfrac{b_0^2}{\beta}& -\tfrac{2b_0d_0}{\beta} \\
 2b_0d_0 &d_0^2-\tfrac{b_0^2}{\beta} \end{bmatrix} & 0\\
  0& \begin{bmatrix}
d_0^2-\tfrac{b_0^2}{\beta} &  -2b_0d_0 \\
 \tfrac{2b_0d_0}{\beta} & d_0^2-\tfrac{b_0^2}{\beta} \end{bmatrix}
   \end{pmatrix}$  respectively.

On the other hand,  according to Lemma \ref{therepsentativesof2},  $ \mathbb{SL}_1(D_{F(\mathbbm{i})})/{[\mathbb{SL}_1(D_{F(\mathbbm{i})}), \mathbb{SL}_1(D_{F(\mathbbm{i})})]}$  together with $T_1=\{ (1,1), (-\alpha, \mathbbm{i}^{-1} ), (\beta, \sqrt{-\beta}^{-1}), ( -\alpha\beta, \mathbbm{i}^{-1 } \sqrt{-\beta}^{-1})\}$ have full image in $\Ua(W_1, \langle, \rangle_1)/{[\Ua(W_1, \langle, \rangle_1), \Ua(W_1, \langle, \rangle_1)]} $. Now let us first describe the action of  $\left\{ (1,1), (-\alpha, \mathbbm{i}^{-1} ), ( \beta, \sqrt{-\beta}^{-1}), ( -\alpha\beta, \mathbbm{i}^{-1 } \sqrt{-\beta}^{-1})\right\}$  on $W_1$ by following the procedure of Section \ref{Satakereallyresults}. The  element $\mathbbm{i}^{-1}$ (resp. $\sqrt{-\beta}^{-1}$)  corresponds to $h_1=\diag(\mathbbm{i}^{-1}, \mathbbm{i}^{-1}, -\mathbbm{i}^{-1}, -\mathbbm{i}^{-1})$ (resp. $h_2=\diag(\sqrt{-\beta}^{-1}, -\sqrt{-\beta}^{-1}, \sqrt{-\beta}^{-1}, -\sqrt{-\beta}^{-1})$) in the algebra $\mathbb{D}_4$ defined in Section \ref{Satakereallyresults}.(Here, $h_1$, $h_2\in \widetilde{K}$.) By  Remark \ref{examplebaisi} $$h_1^{(2)}=\diag(\begin{pmatrix}
-\alpha^{-1}&\\
 & -\alpha^{-1}
 \end{pmatrix}, \begin{pmatrix}
\alpha^{-1}&\\
 & \alpha^{-1}
 \end{pmatrix}, \begin{pmatrix}
\alpha^{-1}&\\
 & \alpha^{-1}
 \end{pmatrix} )$$ and $$h_2^{(2)}=\diag(\begin{pmatrix}
\beta^{-1}&\\
 & \beta^{-1}
 \end{pmatrix}, \begin{pmatrix}
-\beta^{-1}&\\
 & -\beta^{-1}
 \end{pmatrix}, \begin{pmatrix}
\beta^{-1}&\\
 & \beta^{-1}
 \end{pmatrix} ).$$ Therefore the elements $(-\alpha, \mathbbm{i}^{-1})$ and $(\beta, \sqrt{-\beta}^{-1})$ act on $W_1\otimes_FF_1$ as well as $W_1\otimes_FK$ via the matrices $-\alpha P^{-1} h_1^{(2)} P= \diag(I, -I, -I)$ and $\beta P^{-1} h_2^{(2)} P= \diag(I, -I, I)$ of $\GL_3(M_2(F))$ respectively.  Observe that the images of $(-\alpha, \mathbbm{i}^{-1})$ and $(\beta, \sqrt{-\beta}^{-1})$ in $\Sp(W_1, \langle, \rangle_{1,F})$ belong to the center of $\Sp(\mathbb{H}[\mathbbm{j}], \langle, \rangle_{1,F}) \times \Sp\big(\mathbb{H}[\mathbbm{l}], \langle, \rangle_{1,F}\big)$. Next, by Lemma \ref{hermitiandimension1}, $ \mathbb{SL}_1(D_{F(\mathbbm{i})})/{[\mathbb{SL}_1(D_{F(\mathbbm{i})}), \mathbb{SL}_1(D_{F(\mathbbm{i})})]}$ is a cyclic group of order $(q+1)$ generated by $g_0=(d_0^2-\tfrac{b_0^2}{\beta}) -\tfrac{2b_0d_0}{\beta} \sqrt{-\beta}$ in   $\mathbb{SL}_1(F(\sqrt{-\beta}))$. Let $s_0=(d_0^2-\tfrac{b_0^2}{\beta})$, and $t_0= -\tfrac{2b_0d_0}{\beta}$. Then $g_0=s_0 +t_0\sqrt{-\beta}$  corresponds to
$h=\diag( s_0+ t_0 \sqrt{-\beta}, s_0-t_0\sqrt{-\beta}, s_0+ t_0\sqrt{-\beta}, s_0-t_0\sqrt{-\beta})$ in $\mathbb{D}_4$. By Remark \ref{examplebaisi},  $h^{(2)}=\diag( 1, A, 1) \in \GL_3(M_2(K))$ for $A=\begin{pmatrix}
(s_0+ t_0\sqrt{-\beta})^2&0\\
0 & (s_0- t_0\sqrt{-\beta})^2
 \end{pmatrix}\in \GL_2(K)$; it acts on $W_1\otimes_FF_1$ as well as $W_1\otimes_FK$ via the matrix $P^{-1} h^{(2)} P= \diag(1, B, 1)$(lemma \ref{thesubgroupofO}),  for $$B=\begin{pmatrix}
1&\sqrt{-\beta}\\
1 & -\sqrt{-\beta}
 \end{pmatrix}^{-1} \begin{pmatrix}
(s_0+ t_0 \sqrt{-\beta})^2&0\\
0 &(s_0-t_0\sqrt{-\beta})^2
 \end{pmatrix} \begin{pmatrix}
1&\sqrt{-\beta}\\
1 & -\sqrt{-\beta}
 \end{pmatrix}= \begin{pmatrix}
s_0^2- t_0^2 \beta& -2s_0t_0\beta\\
2s_0t_0 & s_0^2- t_0^2 \beta
 \end{pmatrix}.$$ In this way we  understand $g_0$  well as an element of $\textbf{SU}_{V}(F)$, where $V=\mathbb{H}[\mathbbm{i}] \oplus W_1$ is given at the beginning of  Section \ref{Satakereallyresults}.    As shown in Section \ref{Satakereallyresults}, there is a canonical one-to-one  mapping  from $\Ua(W_1, \langle, \rangle_1)$ to $\textbf{SU}_V(F)$. By following the procedure described by  Satake in \cite{Sa}, $g_0$  acts on $W_1=\mathbb{H}[\mathbbm{j}] \oplus \mathbb{H}[\mathbbm{l}]$ as an element  $(s_0^2-t_0^2 \beta +  2s_0t_0\mathbbm{j}, 1)$ in $\Ua(\mathbb{H}[\mathbbm{j}]) \times \Ua(\mathbb{H}[\mathbbm{l}])$. Under the above bases $\mathcal{A}_2^{(1)}$, $\mathcal{A}_2^{(2)}$, and $\mathcal{A}_3$,  such element   acts on  $U_1$, $U_2$, and  $ \mathbb{H}[\mathbbm{l}]$  via the matrices
 $ H_1^{(1)}=
 \begin{bmatrix}
  s_0^2-t_0^2\beta &
4\beta^2 s_0t_0 \\
  -\tfrac{s_0t_0}{\beta} &
  s_0^2-t_0^2\beta \end{bmatrix} $,  $ H_1^{(2)}=
 \begin{bmatrix}
  s_0^2-t_0^2\beta  &   4 \alpha \beta^2 s_0t_0  \\
-\tfrac{s_0t_0}{\alpha\beta} &   s_0^2-t_0^2\beta\end{bmatrix} $, and $ H_2= \begin{pmatrix}
    \begin{bmatrix}
  1& 0\\
  0 & 1 \end{bmatrix} & 0  \\
  0& \begin{bmatrix}
 1& 0\\
  0 & 1\end{bmatrix}
  \end{pmatrix} $ respectively.

Notations being as above, we can  let $\Omega_1= \langle g_0\rangle \times T_1$, and
  $\Omega_2=\langle \mathbbm{e}_{-1}^2 \rangle$. Note that for any two elements $g_1 \in \GU\big( \mathbb{H}[\mathbbm{j}], \langle, \rangle_{1,F} \big)$, $g_2 \in \GU\big( \mathbb{H}[\mathbbm{l}], \langle, \rangle_{1,F}\big)$ with the same similitude factor, the  ordered pair $(g_1, g_2)$  can be viewed as an element in $\GU(  W_1, \langle, \rangle_1)$. Now  let us  define  a subgroup $\Gamma_1$ of $\Gamma$  as  generated by    $\big(\mathbbm{e}_{-1} \mathbbm{j}, \mathbbm{j}; (\mathbbm{e}_{-1} \mathbbm{j})^{-1}\big)$,  $(\mathbbm{e}_{-1} \mathbbm{l}, \mathbbm{l}; \mathbbm{l}^{-1} )$, and $  (a, a; a^{-1})$ for all $ a\in F^{\times}$.
\begin{lemma}
  \begin{itemize}
   \item[(i)] $\Lambda_{\Gamma_1}= F^{\times}=\Lambda_{\Gamma};$
   \item[(ii)] $\Gamma_1 \cap \big( \Ua(W_1, \langle, \rangle_1) \times \Ua(W_2, \langle, \rangle_2)\big)=\left\{ (1, 1), (-1, -1)\right\}$;
   \item[(iii)] $\iota(\Omega_i)^{\iota(\Gamma_1)}=\iota(\Omega_i)$, for $i=1, 2$.
   \end{itemize}
  \end{lemma}
\begin{proof}
Parts (i)(ii) are straightforward. For (iii), when $i=1$, $[\mathbbm{e}_{-1}\mathbbm{j}, \mathbbm{j}](s_0^2-t_0^2 \beta +  2s_0t_0\mathbbm{j})[ \mathbbm{e}_{-1}^{-1}\mathbbm{j}^{-1},  \mathbbm{j}^{-1}]= s_0^2-t_0^2 \beta +  2s_0t_0\mathbbm{j}$. Note that  $\mathbbm{e}_{-1} \mathbbm{l}=\mathbbm{k}$, so $[\mathbbm{k},  \mathbbm{l}](s_0^2-t_0^2 \beta +  2s_0t_0\mathbbm{j})[ \mathbbm{k}^{-1},  \mathbbm{l}^{-1}]= s_0^2-t_0^2 \beta -  2s_0t_0\mathbbm{j}=\mathbbm{e}_{-1}^{-4}\in \Omega_1$.    The result for $i=1$ is similar. 
\end{proof}
 Under the bases $\mathcal{A}_2^{(1)}$, $\mathcal{A}_2^{(2)}$, $\mathcal{A}_3$,  $\big(\mathbbm{e}_{-1} \mathbbm{j}, \mathbbm{j}; (\mathbbm{e}_{-1} \mathbbm{j})^{-1}\big)$ acts on $U_1$, $U_2$, $\mathbb{H}[\mathbbm{l}]$  via   the matrices $A_1^{(1)}=
\begin{bmatrix}
1 &  0\\
0&1 \end{bmatrix}$, $A_1^{(2)}=
\begin{bmatrix}
 d_0^2-\tfrac{b_0^2}{\beta}& -4\alpha \beta b_0d_0\\
 \tfrac{b_0d_0}{\alpha\beta^2} & d_0^2-\tfrac{b_0^2}{\beta}\end{bmatrix}$, $A_2=\begin{pmatrix}
\begin{bmatrix}
d_0& \tfrac{b_0}{\beta}\\
-b_0&d_0\end{bmatrix} &   \\
& \begin{bmatrix}
-d_0& -b_0\\
\tfrac{b_0}{\beta}&-d_0\end{bmatrix}\end{pmatrix}$ respectively.

 Under the bases $\mathcal{A}_2^{(1)}$, $\mathcal{A}_2^{(2)}$, $\mathcal{A}_3$,  $(\mathbbm{e}_{-1} \mathbbm{l}, \mathbbm{l}; \mathbbm{l}^{-1} )$ acts on $U_1$, $U_2$, $\mathbb{H}[\mathbbm{l}]$  via   the matrices $B_1^{(1)}=
\begin{bmatrix}
-d_0&  -2\beta b_0\\
-\tfrac{b_0}{2\beta^2}&d_0 \end{bmatrix}$, $B_1^{(2)}=
\begin{bmatrix}
 d_0& -2\alpha\beta b_0\\
 -\tfrac{b_0}{2\alpha\beta^2}& -d_0\end{bmatrix}$, $B_2=\begin{pmatrix}
\begin{bmatrix}
1& 0\\
0&-1\end{bmatrix} &   \\
& \begin{bmatrix}
1& 0\\
0&-1\end{bmatrix}\end{pmatrix}$ respectively.
\begin{remark}\label{grelations}
\begin{itemize}
\item[(1)] $H_1^{(1)}=\{G_1^{(1)}\}^{2}$, and  $H_1^{(2)}=\{G_1^{(2)}\}^{-2}$;
\item[(2)] $B_1^{(1)}G_1^{(1)}=\{G_1^{(1)}\}^{-1}B_1^{(1)}$;
\item[(3)] $G_1^{(2)} A_1^{(2)}=I$ and $ B_1^{(2)}A_1^{(2)}=\{A_1^{(2)}\}^{-1} B_1^{(2)}$, where $I$ is the identity matrix;
\item[(4)] $\{B_1^{(i)}\}^2=-I$, for $i=1, 2$.
\end{itemize}
\end{remark}
\begin{proof}
(1) Notice that $s_0^2-t_0^2 \beta+ 2s_0t_0 \mathbbm{j}=\mathbbm{e}_{-1}^4$. The matrices $G^{(1)}$, $G^{(2)}$ correspond to the element $\mathbbm{e}_{-1}^4$  acting  on $\mathbb{H}[\mathbbm{j}]$ on the right-hand side,  and    $H_1^{(1)}$ $H_1^{(2)}$ correspond to the element $\mathbbm{e}_{-1}^2$ acting on  $\mathbb{H}[\mathbbm{j}]$  on the left-hand side, so the result follows from $\overline{\mathbbm{e}_{-1}}^{4}=(\mathbbm{e}_{-1})^{-4}$. Other parts (2)---(4) are straightforward.
\end{proof}
Let  $\mathcal {S}_i$ be the subgroup of $\SL(U_i,\langle, \rangle_{1,F}) $ generated by $G_1^{(i)}, H_1^{(i)}, A_1^{(i)}$, and $\mathcal {T}_i= \mathcal {S}_i\cdot \langle B_1^{(i)}\rangle$, for $i=1, 2$ respectively. By Remark \ref{grelations}, $\mathcal {S}_i=\langle G_1^{(i)} \rangle \simeq \mathbb{Z}_{q+1}$, and $\langle B_1^{(i)}\rangle \simeq \mathbb{Z}_4$.  Let us fix a non-trivial character $\psi$ of $F$.
\begin{remark}\label{BCx}
Let $C _1^{(1)}=\begin{bmatrix}
-d_0&  2\beta b_0\\
\tfrac{b_0}{2\beta^2}&d_0 \end{bmatrix}$, and $C_1^{(2)}=
\begin{bmatrix}
d_0& 2\alpha\beta b_0\\
 \tfrac{b_0}{2\alpha\beta^2}& -d_0\end{bmatrix}$. Then 
 \begin{itemize}
\item[(1)]$G_1^{(i)}=B_1^{(i)}C_1^{(i)}$,  $\{G_1^{(i)}\}^{-1}=C_1^{(i)}B_1^{(i)}$, for $i=1,2$;
 \item[(2)] $B_1^{(i)} \notin \mathcal {S}_i$, $C_1^{(i)}\notin \mathcal {S}_i$;
\item[(3)]  $\{C_1^{(i)}\}^2=-I$ , for $i=1, 2$.
 \end{itemize}
 \end{remark}
 \begin{proof}
Parts (1) and (3) are straightforward. For (2),  by observation,  every element  $g\in  \mathcal {S}_i$ has the form $ \begin{bmatrix}
t&  \ast \\
\ast & t \end{bmatrix}$, for some $t\in F$. However the diagonal  parts  of $B_1^{(i)} $, $C_1^{(i)}$ are $\diag(d_0, -d_0)$ or $\diag(-d_0, d_0)$ with  $d_0\neq 0$. 
\end{proof}

 \begin{lemma}\label{thespittingVi}
 The inverse image of $\mathcal {T}_i$ in $\overline{\SL(U_i, \langle, \rangle_{1, F})}$  splits.
 \end{lemma}
 \begin{proof}
 In case $i=1$, we assume  $B_1^{(1)}=u(x) h(t)u(-x)$, for certain $u(x)= \begin{bmatrix}
1&x\\
0&1\end{bmatrix}$, $h(t)= \begin{bmatrix}
0&t\\
-t^{-1}&0\end{bmatrix}$. Then  $C_1^{(1)}=u(-x) h(-t)u(x)$.   Set $Y_1^{\ast} =\Span\{e_1^{\ast}\}$. Let $ c_{Y_1^{\ast}, \psi}$ the Leray cocycle(\emph{cf}. \cite[p.13]{Kud2}) associated to  the Lagrangian subspace $Y_1^{\ast}$ and  $\psi$.  Then using the formulas in \cite[pp.18-21]{Kud2}, we can calculate the Leray cocycle
$$c_{Y^{\ast}, \psi}(B_1^{(1)}, B_1^{(1)})= c_{Y^{\ast},\psi}\big( u(x) h(t)u(-x), u(x) h(t)u(-x)\big)=c_{Y^{\ast}, \psi}\big( h(t), h(t)\big)= \gamma\Big(\psi\circ L\big(Y^{\ast}\cdot h(t), Y^{\ast}, Y^{\ast}\cdot h(t)^2\big)\Big)=1.$$
According to  Theorem \ref{scindagedugroupeR0}, the restriction of $[c_{Y_1^{\ast}, \psi}]$ to $\mathcal {S}_1$  is trivial, so there exists a Borel function $f(t)$ from $\mathcal {S}_1$ to $\mu_8$ such that
$ c_{Y_1^{\ast}, \psi}(t_1, t_2)= f(t_1t_2) f(t_1)^{-1} f(t_2)^{-1}$, for $ t_1, t_2 \in \mathcal{S}_1$.\footnote{ See \cite[p.57]{MVW} for  the  details.}
Now view $f(t)$ as a function from $\mathcal{T}_1$ to $\mu_8$, and define a new $2$-cocycle $\overline{c}_{Y^*, \psi}= c_{Y^*, \psi}\cdot\delta_1(f)$. Then
  $\overline{c}_{Y^{\ast}, \psi}\big([G_1^{(1)}]^k,[G_1^{(1)}]^l\big)=\overline{c}_{Y^{\ast}, \psi}\big(\pm B_1^{(1)}, \pm B_1^{(1)}\big)=1$, for any integers $k, l$.
Moreover,
\begin{equation}\label{equationcocycle}
\begin{split}
\overline{c}_{Y^{\ast}, \psi}(B_1^{(1)}, C_1^{(1)})=c_{Y^{\ast}, \psi}(B_1^{(1)}, C_1^{(1)})=c_{Y^{\ast},\psi}\big( u(x) h(t)u(-x), u(-x) h(-t)u(x)\big)=c_{Y^{\ast},\psi}\big(  h(t)u(x^2), h(-t)\big)\\
=c_{Y^{\ast},\psi}\big(  h(-t)u(x^2), h(t)\big)=c_{Y^{\ast},\psi}\big( u(-x) h(-t)u(x), u(x) h(t)u(-x)\big)=c_{Y^{\ast}, \psi}(C_1^{(1)}, B_1^{(1)})=\overline{c}_{Y^{\ast}, \psi}(C_1^{(1)}, B_1^{(1)})
\end{split}
\end{equation}

   Let us define a new function $h$ from $\mathcal {T}_1$ to $\mu_8$ as
 $h(a)=\overline{c}_{Y^{\ast}, \psi}\big( (B_1^{(1)})^k, (G_1^{(1)})^l\big)$, for $a=(B_1^{(1)})^k(G_1^{(1)})^l\in \mathcal {T}_1$.  Then replace the $2$-cocycle  $\overline{c}_{Y^{\ast},\psi}$ by $c^{\ast}_{Y^{\ast},\psi}=\overline{c}_{Y^{\ast},\psi} \cdot\delta_1(h)$. Next let us apply the criterion(Lemma \ref{onesubgroup}) to the triple $(\mathcal {S}_1 , \langle B_1^{(1)}\rangle, \mathcal {T}_1 )$ instead of $(\Ua(W_1) \times \Ua(W_2), \Gamma_1, \Gamma)$. Go back to the proof of Lemma \ref{onesubgroup}, and it suffices to show that the image  $p\big([c^{\ast}_{Y^{\ast}, \psi}]\big)$ is trivial, where $p: \Ha^2( \mathcal {T}_1 , \mu_8)_1 \longrightarrow \Ha^1\big(\langle  B_1^{(1)} \rangle, \Ha^1( \mathcal {S}_1, \mu_8)\big)$. Since  $\{B_1^{(1)}\}^2=-I$, and $  \mathcal {S}_1$ is a finite cyclic group, it suffices to verify that $p\big([c^{\ast}_{Y^{\ast}, \psi}]\big)( G_1^{(1)},\pm B_1^{(1)})=1$. Note that  $p\big([c^{\ast}_{Y^{\ast}, \psi}]\big)\big( G_1^{(1)}, \pm B_1^{(1)}\big)=\overline{c}_{Y^{\ast}, \psi}\big(G_1^{(1)},  \pm B_1^{(1)}\big)  h(G_1^{(1)}) h(\pm B_1^{(1)}) [h(\pm G_1^{(1)} B_1^{(1)})]^{-1}= \overline{c}_{Y^{\ast}, \psi}\big(G_1^{(1)},  \pm B_1^{(1)}\big) \overline{c}_{Y^{\ast}, \psi}\big( \pm B_1^{(1)}, [G_1^{(1)}]^{-1}\big)^{-1}=1$ by (\ref{equationcocycle}), (\ref{2cyccle}) and Remark \ref{BCx}. The case $i=2$ is similar.
 \end{proof}
 According to  \cite[pp. 245-246]{HanM}, there exists the following morphism on cover groups:
$$\overline{\SL(\Ua_1, \langle, \rangle_{1, F})} \times \overline{\SL(\Ua_2, \langle, \rangle_{1, F})} \times \overline{\Sp(\mathbb{H}[\mathbbm{l}], \langle, \rangle_{1, F})} \longrightarrow \overline{\Sp(W_1, \langle, \rangle_{1, F})}.$$
 Let $Y_3=\Span\{f_1^{\ast}, f_2^{\ast}\}$ be a Lagrangian subspace of $(\mathbb{H}[\mathbbm{l}], \langle, \rangle_{1,F})$. Then the above elements $G_2, H_2$, $A_2, B_2$ all belong to the parabolic subgroup  $P(Y_3)$ of $\Sp\big(\mathbb{H}[\mathbbm{l}], \langle, \rangle_{1,F}\big)$.
By Lemma \ref{onesubgroup} and  Remark \ref{invariant},  we obtain:
\begin{proposition}
In the above two cases I and II, Theorem A holds.
\end{proposition}

\subsection{Case III}\label{Casegnenralcase} In this subsection, we follow  the notions of Section \ref{notations}.  Let $(\mathbbm{i},\mathbbm{j},\mathbbm{l})=(\xi, \varpi, \mathbbm{e}_{-1}
\xi\varpi)$, and  for simplicity we assume  the beginning  $b_0=0$.  Let us fix a  symplectic basis $\mathcal{B}_2=\left\{
-\frac{1}{2\beta}, \frac{\mathbbm{i}}{2\alpha \beta}; \mathbbm{j},
\mathbbm{k}\right\}$ of $(\mathbb{H}[\mathbbm{j}], \langle, \rangle_{1,F})$ and  $\mathcal{B}_3=\left\{ \frac{1}{2},
-\frac{\mathbbm{i}}{2\alpha}; \frac{\mathbbm{l}}{\alpha\beta},
\frac{\mathbbm{i}\mathbbm{l}}{\alpha\beta}\right\}$  of
$(\mathbb{H}[\mathbbm{l}], \langle, \rangle_{1,F})$. It is known that $\Ua(W_2,\langle, \rangle_2)/{[\Ua(W_2, \langle,
\rangle_2), \Ua(W_2, \langle, \rangle_2)]} \simeq
\mathbb{SL}_1(\mathbb{H})/{[\mathbb{SL}_1(\mathbb{H}),
\mathbb{SL}_1(\mathbb{H})]},$
which  is     a cyclic group of order $(q+1)$;  we choose a generator
 with inverse image $x_0+ y_0\mathbbm{i}$ in $\mathbb{SL}_1(F(\mathbbm{i}))$.
 Such element acts on $\mathbb{H}[\mathbbm{j}]$ resp. $\mathbb{H}[\mathbbm{l}]$,
 with respect to the basis $\mathcal{B}_2$ resp. $\mathcal{B}_3$, via the  matrices $L_2=\begin{pmatrix}
  x_0  &   y_0  &   0          &  0    \\
  -y_0 \alpha  & x_0         &  0     & 0           \\
 0   &  0      &       x_0    &  y_0\alpha      \\
  0 &    0    & -y_0 & x_0
 \end{pmatrix}$  resp. $L_3=L_2$.  Now let $Y_{\nu}=\left\{ x+ y\mathbbm{i}\mid x, y\in F\right\}$
  be a Lagrangian subspace of $(\mathbb{H}[\mathbbm{j}], \langle, \rangle_{1,F})$
  as well as $(\mathbb{H}[\mathbbm{l}], \langle, \rangle_{1,F})$, and $Y= Y_{\nu} \oplus Y_{\nu}$
   a Lagrangian subspace of $(W_1, \langle, \rangle_{1,F})$. So  $L_2$ or $L_3$ belongs to the parabolic subgroup $P( Y)$ of $\Sp(W_1, \langle, \rangle_{1,F})$. If we  let $\Omega_2=\mathbb{SL}_1\big(F(\mathbbm{i})\big)$, with full image
 in $\Ua(W_2, \langle, \rangle_2)/[ \Ua(W_2, \langle, \rangle_2), \Ua(W_2, \langle,
 \rangle_2)]$, then the image of $\Omega_2$ in $\Sp( W_1, \langle,
 \rangle_{1, F})$  belongs to $P(Y)$.

Now let us turn to the group $\Ua(W_1, \langle, \rangle_1)$.
Similarly as Cases $I$ $\&$ $II$, the images of  $(-\alpha,
\mathbbm{i}^{-1})$ and $(\beta, \sqrt{-\beta}^{-1})$ in $\Sp(W_1,
\langle, \rangle_{1,F})$ belong to the center  of
$\Sp(\mathbb{H}[\mathbbm{j}], \langle, \rangle_{1,F}) \times
\Sp(\mathbb{H}[\mathbbm{l}], \langle, \rangle_{1,F})$.  By Remark \ref{thequotient12},
$
\mathbb{SL}_1(D_{F(\mathbbm{i})})/[\mathbb{SL}_1(D_{F(\mathbbm{i})}),
\mathbb{SL}_1(D_{F(\mathbbm{i})})]$,  is isomorphic with  a cyclic group of
order $(q^2+1)$. Let us chose a generator of that group  with inverse image   $\alpha_1+ \mathbbm{a}_1\alpha_2$  in $
\mathbb{SL}_1(D_{F(\mathbbm{i})})$ so that   $\alpha_1^{2} +(c_0-
d_0\mathbbm{i}) \alpha_2^{2}=1$. Then  it  acts
on $W_1\otimes_FF_1$ via the  matrix
$g= \begin{pmatrix}
  \alpha_1  & (-c_0+d_0 \mathbbm{i}) \alpha_2     &              &     \\
  \alpha_2  &         \alpha_1          &                &            \\
    &       &       \alpha_1^{\sigma}    &      -\frac{(c_0+d_0\mathbbm{i})}{\mathbbm{i}}\alpha_2^{\sigma}   \\
   &        &  \mathbbm{i}\alpha_2^{\sigma} & \alpha_1^{\sigma}
 \end{pmatrix} \in \mathbb{D}_4.$
  By calculation (\emph{cf}. Remark \ref{examplebaisi}), we know that
  $$g^{(2)}= \begin{pmatrix}
1  &  &                                  &                                  &                                          &                               \\
   &1 &                                  &                                  &                                          &                              \\
   &  & \nnn_{F(\mathbbm{i})/F}(\alpha_1)          & \mathbbm{i} \nnn_{F(\mathbbm{i})/F}(\alpha_2)    & -\frac{c_0+d_0\mathbbm{i}}{\mathbbm{i}} \alpha_1\alpha_2^{\sigma}&   (-c_0+d_0\mathbbm{i}) \alpha_2 \alpha_1^{\alpha}      \\
   &  &\mathbbm{i}\nnn_{F(\mathbbm{i})/F}(\alpha_2)  &   \nnn_{F(\mathbbm{i})/F}(\alpha_1)      & \alpha_2 \alpha_1^{\sigma}   &  \mathbbm{i}\alpha_1 \alpha_2^{\sigma}\\
   &  &\mathbbm{i} \alpha_1\alpha_2^{\sigma} &  (-c_0+ d_0\mathbbm{i}) \alpha_2\alpha_1^{\sigma} & \nnn_{F(\mathbbm{i})/F}(\alpha_1)  & \mathbbm{i}(-c_0+d_0\mathbbm{i}) \nnn_{F(\mathbbm{i})/F}(\alpha_2)  \\
   &  &  \alpha_2\alpha_1^{\sigma}&  -\frac{c_0+d_0\mathbbm{i}}{\mathbbm{i}}\alpha_1 \alpha_2^{\sigma}&  -\frac{c_0+d_0\mathbbm{i}}{\mathbbm{i}} \nnn_{F(\mathbbm{i})/F}(\alpha_2) & \nnn_{F(\mathbbm{i})/F}(\alpha_1)
 \end{pmatrix}.$$
 It can be directly checked that $P^{-1} g^{(2)} P= \begin{pmatrix}
1  & 0 &    0                                   \\
  0 &    X_{22} & X_{23} \\
0 & X_{32} & X_{33}
\end{pmatrix} \in \GL_3(M_2(F))$ for $X_{22}=  \begin{pmatrix}
   \nnn_{F(\mathbbm{i})/F}(\alpha_1)+ \mathbbm{i}\nnn_{F(\mathbbm{i})/F}(\alpha_2)  & 0   \\
 0 &   \nnn_{F(\mathbbm{i})/F}(\alpha_1)- \mathbbm{i}\nnn_{F(\mathbbm{i})/F}(\alpha_2)
\end{pmatrix}$, $ X_{23}= \begin{pmatrix}
(-c_0+d_0\mathbbm{i}) \alpha_2 \alpha_1^{\sigma}+ \alpha_1 \alpha_2^{\sigma}\mathbbm{i} &  0 \\
0 &  -\alpha_2\alpha_1^{\sigma}
\mathbbm{i}-(c_0+d_0\mathbbm{i})\alpha_1 \alpha_2^{\sigma}
\end{pmatrix}$, $X_{32}= \begin{pmatrix}
\alpha_2 \alpha_1^{\sigma}-\frac{(c_0+d_0\mathbbm{i})\alpha_1 \alpha_2^{\sigma}}{\mathbbm{i}}  & 0 \\
0&  \alpha_1\alpha_2^{\sigma} +\frac{(c_0-d_0\mathbbm{i})
\alpha_2\alpha_1^{\sigma}}{\mathbbm{i}}
\end{pmatrix}$,  $X_{33}=\begin{pmatrix}
\nnn_{F(\mathbbm{i})/F}(\alpha_1) +\mathbbm{i}\nnn_{F(\mathbbm{i})/F}(\alpha_2) & 0\\
0&
\nnn_{F(\mathbbm{i})/F}(\alpha_1)-\mathbbm{i}\nnn_{F(\mathbbm{i})/F}(\alpha_2)
   \end{pmatrix}.$ Hence the matrix $\begin{pmatrix}
X_{22} & X_{23}\\
X_{32} & X_{33}
\end{pmatrix}$ corresponds to
$\begin{pmatrix} \nnn_{F(\mathbbm{i})/F}(\alpha_1)+ \mathbbm{i}
\nnn_{F(\mathbbm{i})/F}(\alpha_2)
& (-c_0+d_0\mathbbm{i}) \alpha_2 \alpha_1^{\sigma}+ \alpha_1 \alpha_2^{\sigma}\mathbbm{i}  \\
\alpha_2 \alpha_1^{\sigma}-\frac{(c_0+d_0\mathbbm{i})\alpha_1
\alpha_2^{\sigma}}{\mathbbm{i}}
 & \nnn_{F(\mathbbm{i})/F}(\alpha_1) +\mathbbm{i}\nnn_{F(\mathbbm{i})/F}(\alpha_2)
\end{pmatrix}\in \SL_2(\mathbb{H})$. Taking the set $\Xi_1 =\mathbb{SL}_1\big(F(\mathbbm{i})( \mathbbm{a}_1)\big)$ as defined in Lemma \ref{therepsentativesof2}, we then have
\begin{lemma}\label{Omega1W1}
Recall the set $T_1$  in  Lemma \ref{therepsentativesof2}. Let $\Omega_1$ be the image of $T_1$  in
$\Ua(W_1, \langle, \rangle_1)$. Then:
\begin{itemize}
   \item[(1)]  The composite map
   $ \pm \iota(\Omega_1) \hookrightarrow \iota\big(\Ua(W_1, \langle, \rangle_1)\big)\twoheadrightarrow \iota\big(\Ua(W_1, \langle, \rangle_{1, F})\big)/
   \iota\big([\Ua(W_1, \langle, \rangle_{1, F}), \Ua(W_1, \langle, \rangle_{1, F})]\big)$
   is onto.
   \item[(2)] The image of $\Omega_1$ in $\Sp(W_1, \langle,
   \rangle_{1, F})$ belongs to certain $P(Y)$.
   \end{itemize}
 \end{lemma}
 \begin{proof}
The first statement is immediate. Recall that  $Y_{\nu}=\left\{ x+ y \mathbbm{i}\mid x,y\in F\right\}$ is  a
 Lagrangian subspace of $(\mathbb{H}(\mathbbm{\mathbbm{j}}), \langle, \rangle_{1,F})$
 as well as $(\mathbb{H}[\mathbbm{l}], \langle, \rangle_{1,F})$.
 From the above discussion, we know that the Lagrangian subspace
 $Y=Y_{\nu} \oplus Y_{\nu}$  of $(W_1, \langle, \rangle_{1,F})$ is $(\alpha_1+ \mathbbm{a}_1 \alpha_2)$-stable for $\alpha_1 +\mathbbm{a}_1 \alpha_2 \in \mathbb{SL}_1\big( F(\mathbbm{i})(\mathbbm{a}_1)\big)$,
 which is the required  result.
\end{proof}

We assume $(\mathbbm{e}_{-1})^{-1}=d_0 -\frac{c_0}{\alpha} \mathbbm{i}$, and $\mathbbm{l}= (\mathbbm{e}_{-1})^{-1} \mathbbm{k}$.  Now we let $\Gamma_1$ be a subgroup of $\Gamma$ generated by the following elements: (i)  $(a, a; a^{-1})$ for all $a \in F^{\times }$, (ii) $( \mathbbm{i}, \mathbbm{i};
   \mathbbm{i}^{-1} \mathbbm{e}_{-1}^{-1})$, (iii) $( \mathbbm{e}_{-1}\mathbbm{l} , \mathbbm{l};
   \mathbbm{l}^{-1})$. Then
(1)  $ \Lambda_{\Gamma_1}=F^{\times}=\Lambda_{\Gamma}$,
 (2)  $ \iota(\Gamma_1)\cap \iota\big( \Ua(W_1, \langle, \rangle_1)\times  \Ua(W_2, \langle, \rangle_2) \big)= 1$.

\begin{lemma}\label{theconditionC46}
\begin{itemize}
\item[(1)] The condition  $(\Ca4)$ of Section \ref{reductions}  holds in
this case.
\item[(2)] Under the
restriction $\Ha^{2}(\iota(\Gamma), \mu_8) \longrightarrow
\Ha^2(\iota(\Gamma_1), \mu_8),$
 the image of $[c]$ is  trivial.
 \end{itemize}
\end{lemma}
\begin{proof}
Let $Y=\left\{ [ (x_1+ y_1 \mathbbm{i}), (x_2+ y_2 \mathbbm{i})]\mid x_i, y_i \in F\right\}$ be a Lagrangian subspace of $(W_1, \langle, \rangle_{1, F})$. Then   $\iota(\Gamma_1)$, $\iota(\Omega_1)$, $\iota(\Omega_2)$ all belong to $P(Y)$.
\end{proof}
Finally we achieve the main result of this subsection:
\begin{proposition}
In  Case III, Theorem A holds.
\end{proposition}

\section{The proof of  the main theorem-Part IV.}\label{TheproofofthemaintheoremIV}
Let $\big(W_1=\mathbb{H}[\mathbbm{i}]\oplus  \mathbb{H}[\mathbbm{j}] \oplus \mathbb{H}[\mathbbm{l}], \langle, \rangle_1\big)$ be a right anisotropic shew hermitian space over $\mathbb{H}$ of dimension $3$ given  at the beginning of Section \ref{TheproofofthemaintheoremIII}. Let $(W_2=\mathbb{H}, \langle,\rangle_2)$ be a left hermitian space over $\mathbb{H}$ of dimension $1$. Let $ \left\{ 1, \xi, \varpi, \xi\varpi\right\}$ be the fixed standard basis of $\mathbb{H}$ as given in Section \ref{notationss}. By Proposition \ref{propositionforskewher}, we assume that $(\mathbbm{i},\mathbbm{j}, \mathbbm{l})=(\xi, \varpi, \mathbbm{e}_{-1} \xi \varpi)$, and  $b_0=0$. Let $\mathbbm{e}_{-1}$ be the Teichm\"uller representative of an element of $k_{F(\mathbbm{i})}$ in $\mathfrak{O}_{F(\mathbbm{i})}$ with order $2(q+1)$.

As before we endow $W_1$ with the $F$-symplectic form $\langle, \rangle_{1,F}=\Trd(\langle, \rangle_1)$ so that the canonical mapping
$\theta: W =W_1 \otimes_{\mathbb{H}}W_2  \simeq \Big( \mathbb{H}[\mathbbm{i}] \oplus \mathbb{H}[\mathbbm{j}] \oplus \mathbb{H}[\mathbbm{l}]\Big) \otimes_{\mathbb{H}} \mathbb{H} \longrightarrow \mathbb{H}[\mathbbm{i}] \oplus \mathbb{H}[\mathbbm{j}] \oplus \mathbb{H}[\mathbbm{l}];$
$ w_1 \otimes \mathbbm{d} \longmapsto w_1 \mathbbm{d}$,
defines an isometry between $(W, \langle, \rangle)$ and $(W_1, \langle, \rangle_{1,F})$.  We fix a basis $\mathcal{B}_1=\left\{ -\frac{1}{2\alpha}, \frac{\mathbbm{j}}{2\alpha \beta}; \mathbbm{i}, -\mathbbm{k}\right\}$ of $(\mathbb{H}[\mathbbm{i}], \langle, \rangle_{1,F})$, resp. $\mathcal{B}_2=\left\{ -\frac{1}{2\beta}, \frac{\mathbbm{i}}{2\alpha\beta}; \mathbbm{j}, \mathbbm{k}\right\}$ of $(\mathbb{H}[\mathbbm{j}], \langle, \rangle_{1,F})$, resp. $\mathcal{B}_3=\left\{ \frac{1}{2}, -\frac{\mathbbm{i}}{2\alpha}; \frac{\mathbbm{l}}{\alpha\beta}, \frac{\mathbbm{i}\mathbbm{l}}{\alpha\beta}\right\}$ of $(\mathbb{H}[\mathbbm{l}], \langle,\rangle_{1,F})$. In this case $\Ua(W_2, \langle, \rangle_2) \simeq \mathbb{SL}_1(\mathbb{H})$. According to the discussion in Section \ref{Casegnenralcase},  the image of $\Omega_2=\langle \mathbbm{e}_{-1}^2\rangle$ in $\Ua(W_2, \langle, \rangle_2)/[ \Ua(W_2, \langle, \rangle_2), \Ua(W_2, \langle, \rangle_2)]$ is full.

Next let us consider $\Ua(W_1, \langle, \rangle_1)$. By Lemma \ref{therepsentativesof}, the group $\Ua(W_1, \langle, \rangle_1)/{[\Ua(W_1, \langle, \rangle_1), \Ua(W_1, \langle, \rangle_1)]}$ is generated by the images of $\mathbb{SL}_1(\mathbb{D}_4)/[\mathbb{SL}_1(\mathbb{D}_4), \mathbb{SL}_1(\mathbb{D}_4)]$ and
$\left\{ (1,1), (-\alpha, \mathbbm{i}^{-1}), (-\beta, \sqrt{-\beta}^{-1}), (\alpha \beta, \mathbbm{i}^{-1}\sqrt{-\beta}^{-1})\right\}$. According to the proof of Case III in Section \ref{Casegnenralcase}, the images of $(-\alpha, \mathbbm{i}^{-1})$, $(-\beta, \sqrt{-\beta}^{-1})$ in $\Sp(W_1, \langle, \rangle_{1,F})$ belong to the center of $\Sp(\mathbb{H}[\mathbbm{i}], \langle, \rangle_{1,F}) \times \Sp(\mathbb{H}[\mathbbm{j}], \langle, \rangle_{1,F}) \times \Sp(\mathbb{H}[\mathbbm{l}], \langle, \rangle_{1,F})$. On the other hand, $\mathbb{SL}_1(\mathbb{D}_4)/{[\mathbb{SL}_1(\mathbb{D}_4), \mathbb{SL}_1(\mathbb{D}_4)]}$ is isomorphic with $\mathbb{SL}_1(\mathbb{D}_{F(i)})/{[\mathbb{SL}_1(\mathbb{D}_{F(i)}), \mathbb{SL}_1(\mathbb{D}_{F(i)})]}$; taking the set $\Xi_1$ as defined  in Lemma \ref{therepsentativesof2}, we then have:
\begin{lemma}
Recall the set  $T$ in Lemma \ref{therepsentativesof}.  Let $\Omega_1$ be the image of $T$ in $\Ua(W_1, \langle, \rangle_{1,F})$. Then:
\begin{itemize}
   \item[(1)]  The composite map
   $\pm\iota(\Omega_1) \hookrightarrow \iota\big(\Ua(W_1, \langle, \rangle_1)\big)\twoheadrightarrow \iota\big(\Ua(W_1, \langle, \rangle_{1, F})\big)/
   \iota\big([\Ua(W_1, \langle, \rangle_{1, F}), \Ua(W_1, \langle, \rangle_{1, F})]\big)$
   is onto.
   \item[(2)] The image of $\Omega_1$ in $\Sp(W_1, \langle,
   \rangle_{1, F})$ belongs to  $P(Y)$, for certain Lagrangian subspace of $(W_1, \langle, \rangle_{1,F})$.
   \end{itemize}
   \end{lemma}
   \begin{proof}
   We only sketch the proof of the second statement. Let us follow the notations in Section \ref{Satakereallyresults}.  By the arguments of Section \ref{Casegnenralcase}, an element $\alpha_1 +   \mathbbm{a}_1 \alpha_2 \in \mathbb{SL}_1\big( F(\mathbbm{i})(\mathbbm{a}_1)\big)$  acts on $W_1$ via the matrix with the form $\begin{pmatrix}
1   & 0 & 0\\
0   &X_{22} & X_{23}\\
0  & X_{32} & X_{33}
\end{pmatrix} \in \GL_3(M_2(F))$, where $\begin{pmatrix}
X_{22} & X_{23}\\
X_{32} & X_{33}
\end{pmatrix}$ is given in Lemma \ref{Omega1W1}; then the result is clear.
 \end{proof}
 Similarly as before, we assume $(\mathbbm{e}_{-1})^{-1}= d_0 -\frac{c_0}{\alpha} \mathbbm{i}$, and $\mathbbm{l}= (\mathbbm{e}_{-1})^{-1} \mathbbm{k}$. Let $\Gamma_1$ be a subgroup of $\Gamma$ generated by $\big(\mathbbm{e}_{-1} \mathbbm{i}, \mathbbm{i}, \mathbbm{i};  \mathbbm{i}^{-1}(\mathbbm{e}_{-1})^{-1}\big)$, $(\mathbbm{e}_{-1} \mathbbm{l}, \mathbbm{e}_{-1}\mathbbm{l}, \mathbbm{l}; \mathbbm{l}^{-1} )$, and $ (a, a, a; a^{-1}) $  for all $ a\in F^{\times}$.

   \begin{lemma}
  \begin{itemize}
  \item[(1)] $\Lambda_{\Gamma_1}= F^{\times} =\Lambda_{\Gamma}$.
  \item[(2)] $ \iota(\Gamma_1)\cap \iota\big( \Ua(W_1, \langle, \rangle_1)\times  \Ua(W_2, \langle, \rangle_2) \big)= 1$.
  \end{itemize}
  \end{lemma}
  \begin{proof}
  Straightforward.
  \end{proof}
 Let $\mathcal{T}_2$ be the subgroup of $\Sp(W_1, \langle,\rangle_{1,F})$, generated by $\iota(\Gamma_1)$  and $\iota(\Omega_2)$.
  \begin{lemma}\label{thetrivialresult33}
The restriction of $[c]$ to $\mathcal{T}_2$  is trivial.
 \end{lemma}
 \begin{proof}
By definition,  $\mathcal{T}_2$  is a subgroup of
$\Sp\big(\mathbb{H}[\mathbbm{i}], \langle, \rangle_{1,F}\big) \times \Sp\big(\mathbb{H}[\mathbbm{j}], \langle, \rangle_{1,F}\big) \times \Sp\big(\mathbb{H}[\mathbbm{l}], \langle, \rangle_{1,F}\big);$
we denote its image in the first group by $\mathcal{T}_2^{(1)}$ , the second one by $\mathcal{T}_2^{(2)}$, and the third one by $\mathcal{T}_2^{(3)}$. By what we have proved in Lemma \ref{thespittingVi}, $\overline{\mathcal{T}_2^{(1)}}$  splits. We now let $Y_{\ast}=\left\{ x+ y \mathbbm{i}\mid x, y \in F\right\}$ be a Lagrangian subspace of $(\mathbb{H}[\mathbbm{j}], \langle, \rangle_{1,F})$ as well as $(\mathbb{H}[\mathbbm{l}], \langle, \rangle_{1,F})$. Then $\mathcal{T}_2^{(2)}$, $\mathcal{T}_2^{(3)}$  both belong to $P(Y_{\ast})$  so the lemma is proved.
 \end{proof}
 By Lemma \ref{onesubgroup}, Remark \ref{invariant},  we obtain
 \begin{proposition}
Under the  conditions of the beginning, Theorem A holds.
\end{proposition}

\section{The proof of  the main theorem-Part V.}\label{TheproofofthemaintheoremV}
 Let $(H, \langle, \rangle)$ be a right skew hermitian hyperbolic plane over $\mathbb{H}$ defined as in Section \ref{thehyperbolicunitarygroups}.  Let $( W_1 = \mathbb{H}[\mathbbm{i}] \oplus H, \langle, \rangle_1)$ be a  right skew hermitian hyperbolic space over $\mathbb{H}$ of dimension $3$. Let $(W_2= \mathbb{H}, \langle, \rangle_2)$ be a left hermitian space over $\mathbb{H}$ of dimension $1$.  Let $\left\{ 1, \mathbbm{i}, \mathbbm{j}, \mathbbm{i}\mathbbm{j}=\mathbbm{k}=-\mathbbm{j}\mathbbm{i}\right\}$ be a standard basis of $\mathbb{H}$. We endow $W_1$ with  the $F$-symplectic form $\langle, \rangle_{1,F}=\Trd(\langle, \rangle_1)$. Then it can be checked that the canonical mapping
$\theta: W=W_1 \otimes_{\mathbb{H}} W_2 \longrightarrow W_1=\mathbb{H}[\mathbbm{i}] \oplus H;$
$  w_1 \otimes \mathbbm{d} \longmapsto w_1 \mathbbm{d}$,
defines an isometry between $(W, \langle, \rangle)$ and $(\mathbb{H}[\mathbbm{i}] \oplus H, \langle, \rangle_{1,F})$.   Let us fix   a  complete polarisation  $H=X \oplus X^{\ast}$ of $H$. By Lemmas  \ref{canonicalsurjectivemap}, \ref{Horth}, there is a surjective composite  map
$$\mathfrak{h}: \frac{\mathbb{H}^{\times}}{[\mathbb{H}^{\times}, \mathbb{H}^{\times}]} \longrightarrow \frac{\Ua(H, \langle, \rangle_1)}{[\Ua(H, \langle, \rangle_1), \Ua(H, \langle, \rangle_1)]} \longrightarrow \frac{\Ua(W_1, \langle, \rangle_1)}{[\Ua(W_1, \langle, \rangle_1), \Ua(W_1, \langle, \rangle_1)]}.$$
\begin{corollary}
Let $\Omega_1= \mathbb{H}^{\times}$, and $Y= Y_1 \oplus X^{\ast}$, for an arbitrary Lagrangian subspace $Y_1$ of $ (\mathbb{H}[\mathbbm{i}], \langle, \rangle_{1,F})$. Then the image of $\Omega_1$ in $\frac{\Ua(W_1, \langle, \rangle_1)}{[\Ua(W_1, \langle, \rangle_1), \Ua(W_1, \langle, \rangle_1)]}$ is full and  $\iota(\Omega_1)\subseteq P(Y)$.
\end{corollary}
 For the  group $\Ua(W_2, \langle, \rangle_2)$, we can  let $\Omega_2=\mathbb{SL}_1( F(\xi))$.  Then it is clear that the image of $\Omega_2$ in $\frac{\Ua(W_2, \langle, \rangle_2)}{[\Ua(W_2, \langle, \rangle_2), \Ua(W_2, \langle, \rangle_2)]}$ is full. Following the comprehensive discussion in Section \ref{TheproofofthemaintheoremIV}, we define a subgroup $\Gamma_1$ of $\Gamma$ as follows:

(I)  If $ \mathbbm{i}=\xi$ and $\mathbbm{e}_{-1} \in F(\mathbbm{i})$,  then let $\Gamma_1$ be  generated by   $\big(\mathbbm{e}_{-1}\mathbbm{i}, \diag(\mathbbm{e}_{-1}\mathbbm{i}, \mathbbm{e}_{-1}\mathbbm{i});  \mathbbm{i}^{-1}\mathbbm{e}_{-1}^{-1}\big)$, $ \big(\mathbbm{j}, \diag(\mathbbm{e}_{-1}\mathbbm{j}, \mathbbm{e}_{-1}\mathbbm{j}); \mathbbm{j}^{-1}\mathbbm{e}_{-1})$, $ \big(a,  \diag(a, a);   a^{-1}\big)$ for all  $ a\in F^{\times}$;

(II) If $ (\mathbbm{i}, \mathbbm{j})=(\varpi,\xi)$ or $(\xi \varpi, \xi)$,  and $\mathbbm{e}_{-1} \in F(\mathbbm{j}) $,  then let $\Gamma_1$ be generated by   $ \big(\mathbbm{i},  \diag(\mathbbm{i}, \mathbbm{i}); \mathbbm{i}^{-1}\big)$, $\big(\mathbbm{j}, \diag(\mathbbm{e}_{-1}\mathbbm{j}, \mathbbm{e}_{-1}\mathbbm{j}); \mathbbm{j}^{-1}\mathbbm{e}_{-1}^{-1}\big)$, $ \big(a, \diag(a, a); a^{-1}\big)$  for all $ a \in F^{\times}$.

\begin{lemma}\label{threeres}
\begin{itemize}
\item[(1)] $\Lambda_{\Gamma_1}= F^{\times} =\Lambda_{\Gamma}$;
\item[(2)] $\Gamma_1\cap \big( \Ua(W_1, \langle, \rangle_1) \times \Ua(W_2, \langle, \rangle_2)\big)=\left\{ (-1, -1), (1, 1)\right\}$;
\item[(3)] $\iota(\Omega_i)^{\iota(\Gamma_1)}=\iota(\Omega_i)$, for $i=1, 2$.
\end{itemize}
\end{lemma}
\begin{proof}
Straightforward.
\end{proof}
In case (I), we  let $\mathcal{T}_2$ be the subgroup of $\Sp(W_1, \langle,\rangle_{1,F})$, generated by $\iota(\Gamma_1)$  and $\iota(\Omega_2)$. Analogous to Lemma \ref{thetrivialresult33}, it can be also shown that
the restriction of $[c]$ to $\mathcal{T}_2$  is trivial.  In case (II), we let $Z_1=\{ x_1+ y_1 \mathbbm{j}\mid x_1, y_1 \in F\}$ be a Lagrangian subspace of $(\mathbb{H}[\mathbbm{i}], \langle, \rangle_{1, F})$. Then $\iota(\Gamma_1)$,  $\iota(\Omega_2)\subseteq P(Z_1 \oplus X^{\ast})$. Finally  by Lemma \ref{onesubgroup}, Remark \ref{invariant}, we  obtain
 \begin{proposition}
Under conditions  at the beginning of this section, Theorem A holds.
\end{proposition}

\section{The proof of  the main theorem-Part VI.}\label{TheproofofthemaintheoremVI}
In this section we  finish proving Theorem A in the general case.  The whole process has already been done  in Section \ref{TheproofofthemaintheoremV}. To avoid duplicating work, we only give a sketch of  the necessary steps.  We will use the notations introduced in Section \ref{statementofresult}. Now let $W_{\nu}=W_{\nu}^0 \oplus H_{\nu}$ be a Witt's decomposition with $W_{\nu}^0$ being its anisotropic subspace, and $H_{\nu} \simeq m_{\nu} H$ being its hyperbolic  subspace, as $\nu$ runs through $1$, $2$.
\begin{remark}
\begin{itemize}
\item[(1)] If $W_1^0=0$, or $W_2^0=0$, we deduce the result from Section \ref{TheproofofthemaintheoremI}.
\item[(2)] If $m_1=0=m_2$, the result has been verified  in Sections \ref{TheproofofthemaintheoremII}, \ref{TheproofofthemaintheoremIII}, \ref{TheproofofthemaintheoremIV}.
\item[(3)] The case $D=F$, $\{\epsilon_1, \epsilon_2\}= \{ \pm 1\}$ has been completely discussed in Corollary \ref{scindagesursymplectique}. In what follows we shall exclude  this case automatically.
\end{itemize}
\end{remark}
Suppose now $W_1^0 \neq 0$, $W_2^0 \neq 0$, and we assume that either $m_1$ or $m_2$ is nonzero. In this situation there is a morphism
 $ i: \Sp(W_1^0 \otimes_D W^0_2) \times \Sp(W_1^0 \otimes_D  H_2) \times \Sp(H_1 \otimes_D W_2^0) \times \Sp(H_1 \otimes_D  H_2) \longrightarrow \Sp(W_1 \otimes_DW_2),$
 which  induces a morphism on cover groups by \cite[pp. 245-246]{HanM}, that is,
$\overline{i}: \overline{\Sp(W_1^0 \otimes_D W^0_2)}\times \overline{\Sp(W_1^0 \otimes_D  H_2)} \times \overline{\Sp(H_1 \otimes_D W_2^0)} \times \overline{\Sp(H_1 \otimes_D H_2)} \longrightarrow \overline{\Sp(W_1 \otimes_DW_2)}.$
For these subspaces
$W_1^0 \otimes_D W_2^0, W_1^0 \otimes_D H_2, H_1 \otimes_D W_2^0, H_1 \otimes_D H_2$
 of $W_1 \otimes_D W_2$, we have already defined the suitable pairs
 $$(\Gamma^{(0, 0)}, \Gamma_1^{(0, 0)}), (\Gamma^{(H_1, 0)}, \Gamma_1^{(H_1, 0)}), (\Gamma^{(0, H_2)}, \Gamma_1^{(0, H_2)}), (\Gamma^{(H_1, H_2)}, \Gamma_1^{(H_1, H_2)})$$
  of the subgroups of $\big(\GU(W_1^0,\langle, \rangle_1)  \times  \GU(W_2^0, \langle, \rangle_2)\big)$, $\cdots$, $\big(\GU(H_1, \langle, \rangle_1)\times \GU(H_1, \langle, \rangle_2)\big)$ respectively in Sections \ref{TheproofofthemaintheoremI}, \ref{TheproofofthemaintheoremII}, \ref{TheproofofthemaintheoremIII}, \ref{TheproofofthemaintheoremIV}, \ref{TheproofofthemaintheoremV}.  By Proposition \ref{thespittingcases}, we can let $\Gamma_1$ be a subgroup of $\Gamma$ consisting of the elements
$[(g_1^{(0)}, g_1^{(H_1)}), (g_2^{(0)}, g_2^{(H_2)})]$
such that (1) $(g_1^{(0)}, g_2^{(0)}) \in \Gamma_1^{(0, 0)}$,
 (2) $(g_1^{(0)}, g_2^{(H_2)}) \in \Gamma_1^{(0, H_2)}$,
 (3) $(g_1^{(H_1)}, g_2^{(0)}) \in \Gamma_1^{(H_1, 0)}$,  and consequently
(4)  $(g_1^{(H_1)}, g_2^{(H_2)}) \in \Gamma_1^{(H_1, H_2)}$.
As a consequence we obtain:
\begin{lemma}\label{thetrivialresult123}
\begin{itemize}
\item[(1)] $\Lambda_{\Gamma}= \Lambda_{\Gamma_1}$.
\item[(2)] $\iota(\Gamma_1) \cap \iota\big(\Ua(W_1, \langle, \rangle_1) \times \Ua(W_2, \langle,\rangle_2)\big)=1$.
\end{itemize}
 \end{lemma}

 \begin{lemma}\label{thetrivialresult133}
Under the restriction
$\Ha^{2}(\iota(\Gamma), \mu_8) \longrightarrow \Ha^2(\iota(\Gamma_1), \mu_8),$
 the image of $[c]$ is  trivial.
 \end{lemma}

\subsection{}Suppose now $m_1m_2 \neq 0$.  By Lemmas \ref{canonicalsurjectivemap}, \ref{Horth}, for each $\nu=1, 2$, there exists  a surjective composite map
$$\frac{D^{\times}}{[D^{\times}, D^{\times}]} \twoheadrightarrow \frac{\Ua(H, \langle, \rangle_{\nu})}{[ \Ua(H, \langle, \rangle_{\nu}), \Ua(H, \langle, \rangle_{\nu})]} \twoheadrightarrow \frac{\Ua(W_{\nu}, \langle, \rangle_{\nu})}{[\Ua(W_{\nu}, \langle, \rangle_{\nu}),  \Ua(W_{\nu}, \langle, \rangle_{\nu})]}.$$
Namely we choose $\Omega_{\nu}=D^{\times}$.   Let $Y_{\nu}= Y_{\nu}^{(0)}\oplus Y_{\nu}^{(H_{\nu})}$ be  a Lagrangian subspace of  $\big(W_1 \otimes W_2, \langle, \rangle_1 \otimes \tau\big( \langle, \rangle_2\big)\big)$ consisting of an arbitrary Lagrangian subspace $Y_{\nu}^{(0)}$ of $W_1^0 \otimes_DW_2^0$, and $Y_{\nu}^{(H_{\nu})}$ as defined in Proposition \ref{thespittingcases} in each  case. By Proposition \ref{thespittingcases} we obtain
\begin{lemma}\label{thelastcasep12}
\begin{itemize}
\item[(1)] $\iota(\Omega_{\nu}) \subseteq P(Y_{\nu})$.
 \item[(2)]  $\iota(\Omega_{\nu})^{\iota(\Gamma_1)} \subseteq P(Y_{\nu})$.
 \end{itemize}
\end{lemma}
By Lemma \ref{onesubgroup}, Theorem A holds in this  case.

\subsection{}Without loss of generality, suppose now $m_1\neq 0$ and $m_2=0$. In this case, we let    $\Omega_{1}= D^{\times}$,  and $Y_1$ be the Lagrangian subspace of $W_1 \otimes W_2$ as defined before. Nevertheless, $W_2=W_2^0$ is an anisotropic $\epsilon$-hermitian space over $D$. We follow the discussion in Sections \ref{TheproofofthemaintheoremI}, \ref{TheproofofthemaintheoremII}, \ref{TheproofofthemaintheoremIII}, \ref{TheproofofthemaintheoremIV}, \ref{TheproofofthemaintheoremV}, and  define  the distinct set $\Omega_2$ in each case.  With the benefit, we obtain the same result as Lemma \ref{threeres}. We can  define a new subgroup $\mathcal{T}_2$ of $\Sp(W_1 \otimes W_2)$ as generated by $\iota(\Gamma_1)$ and $\iota(\Omega_2)$. Then similarly as Section \ref{TheproofofthemaintheoremV}, it can be shown that either $\iota(\Omega_2)^{\iota(\Gamma_1)} \subseteq P(Y)$ for certain Lagrangian subspace of  $\big(W_1 \otimes W_2, \langle, \rangle_1 \otimes \tau\big( \langle, \rangle_2\big)\big)$, or   the restriction of $[c]$ to $\mathcal{T}_2$ is trivial.  Without a doubt, Theorem A holds in this case.

\labelwidth=4em
\addtolength\leftskip{25pt}
\setlength\labelsep{0pt}
\addtolength\parskip{\smallskipamount}

\end{document}